\documentclass[
11pt,                          
draft,                         
english                        
]{article}

\usepackage[english]{babel}    
\usepackage{amsmath}           
\usepackage[utf8]{inputenc}    
\usepackage[T1]{fontenc}       
\usepackage{longtable}         
\usepackage{exscale}           
\usepackage[final]{graphicx}   
\usepackage[sort]{cite}        
\usepackage{array}             
\usepackage{wasysym}           
\usepackage[a4paper]{geometry} 
\usepackage[multiuser]{fixme}  
\usepackage{xspace}            
\usepackage{tikz}              
\usepackage{ifdraft}           
\usepackage{chairx}            
\usepackage[expansion=false    
           ]{microtype}        
\usepackage[nottoc]{tocbibind} 
 \usepackage[backref=page,      
            final=true,         
           pdfpagelabels       
           ]{hyperref}         
\usepackage[all]{xy}
\usepackage{slashed}
\usepackage{faktor}
\usepackage{amssymb}
\usetikzlibrary{matrix}
\usepackage{tikz-cd}

\usepackage{todonotes}

\usepackage{bbm}

\ifdraft{\synctex=1}{}

\fxusetheme{color}

\newcommand{\bla}{\langle \hspace{-2.7pt} \langle}
\newcommand{\bra}{\rangle\hspace{-2.7pt} \rangle}

\title{Normal Forms for Dirac-Jacobi bundles and Splitting Theorems for Jacobi Structures}
\date{}
\author{
\textbf{Jonas Schnitzer}\thanks{\texttt{jschnitzer@unisa.it}}\\[0.5cm]
 Dipartimento di Matematica \\
  Università degli Studi di Salerno \\
  Via Giovanni Paolo II, 132\\
  84084  Fisciano (SA) \\
  Italy}

\begin{document}

\maketitle

\begin{abstract}
The aim of this paper is to prove a normal form Theorem for Dirac-Jacobi bundles using the recent techniques from 
\cite{2016arXiv160505386B}. As the most important consequence, we can prove the splitting theorems of Jacobi pairs 
which was proposed by Dazord, Lichnerowicz and Marle in \cite{SplittingThmJac}. As an application we provide a alternative
proof of the splitting theorem of homogeneous Poisson structures.  
\end{abstract}

\tableofcontents

\section{Introduction}

Since the work of Weinstein \cite{weinstein1983}, in which he proved his famous local splitting theorem for Poisson 
manifolds, many works appeared concerning different viewpoints on the proof and even give more general statements, namely 
normal form theorems. Frejlich and Marcut proved a normal form theorem around Poisson (cosymplectic) transversals of Poisson 
manifolds in \cite{FM2017}. In \cite{FM2018} they used the techniques of Dual Pairs to prove a similar statement for Dirac 
structures. And finally, there is a unified approach by Bursztyn, Lima and Meinrenken in \cite{2016arXiv160505386B} to prove 
normal forms for Poisson related structures. 

Jacobi geometry was introduced by Kirillov in \cite{K1976} as local Lie algebras. They have a deep connection to Poisson 
geometry, since every Poisson structure defines a Jacobi bracket. Moreover, every Jacobi structure induces a Poisson 
structure on a manifold of one dimension more, this is known as the symplectization or homogenezation, see \cite{BGG2017} and 
its references for a detailed discussion.  
In Jacobi geometry there is also a local splitting theorem available, which was proven  
by Dazord, Lichnerowicz and Marle in \cite{SplittingThmJac}. Nevertheless, after this work the parallels in the work of Poisson and Jacobi geometry stopped, at least in the context of local structure. The aim of this paper is to fill these 
gaps, prove normal form theorems for Jacobi bundles and give a more intrinsic proof of the splitting theorems. 
To do so, we will chose the approach of \cite{2016arXiv160505386B} and start with so-called Dirac-Jacobi bundles which generalize the notion of Jacobi structures.

Dirac-Jacobi bundles were introduced in \cite{DirJacBun} by Vitagliano and are a slight generalization of Wade's 
$\mathcal{E}^1(M)$-Dirac structures (see \cite{Wade2000}). Moreover, these bundles are a Dirac theoretic generalizations of 
Jacobi bundles, as usual Dirac structures are for Poisson manifolds.

We want to stress that the methods, which are expressed in this note are also 
suitable for proving splittings for involutive fat anchored vector bundles $(E,L\to M,\rho)$, i.e. a vector bundle $E\to M$, 
a line bundle $L\to M$ and a bundle map $\rho\colon E\to DL$, such that $\Secinfty(\rho(E))$ is closed with respect to the 
bracket, as well as Jacobi-algebroids (see \cite{2017arXiv170508962T}). We do not want to treat that in detail since every 
involutive fat anchored vector bundle is in particular, by composing the anchor $\rho$ with the anchor of $DL$, an involutive 
anchored vector bundle and can be treated with the methods in \cite{2016arXiv160505386B}. The same holds true for Jacobi- 
algebroids.  
 
This short note is organized as follows: we recall the necessary structures in order to define the setting for Dirac-Jacobi 
structures, the omni-Lie algebroid of a line bundle (see \cite{CHEN2010799}) in Section \ref{Sec: Not}.
Afterwards, we introduce the notion of 
Euler-like derivations, which are the crucial ingredient for the proofs of the main theorems. After this we are able to 
provide a normal form theorem for Dirac-Jacobi bundles, which is the main part of Section \ref{Sec: NormFroDJ}. In the following section, we want to apply this normal form theorem 
to the special case of Jacobi bundles, which allows us to state and prove two normal form theorems for Jacobi bundles, which 
allow us to give a different prove of the splitting theorems of Jacobi pairs, first provided in \cite{SplittingThmJac}. 
Moreover, we can apply this theorems to provide a splitting theorem for homogeneous Poisson structure around points where 
the homogeneity does not vanish, which was also done in \cite{SplittingThmJac}. 
Note that in \cite{SplittingThmJac} the proof works exactly the other way around: they prove
a local splitting of homogeneous Poisson structures and use it to prove the splitting of Jacobi structures. 

\noindent
\textbf{Acknowledgements:} I would like to thank my advisor, Luca Vitagliano, who suggested me this project and helped me a 
lot in turning it in to a paper as well as Chiara Esposito who helped me to improve the presentation. 
The content of  this
note was  produced almost completely during a stay at IMPA in Rio de Janeiro from April to July in 2018, where I was warmly 
received in the Poisson Geometry group. In particular,  I would like to thank Henrique Bursztyn for discussions and useful 
suggestions.

\section{Preliminaries and Notation}\label{Sec: Not}
This introductory section is divided into two parts: first we recall the Atiyah algebroid of a vector bundle and the 
corresponding $Der$-complex with applications to contact and Jacobi geometry. Afterwards, we introduce the arena for 
the so-called Dirac-Jacobi bundles in odd dimensions, the omni-Lie algebroids, and give a quick reminder of Dirac-Jacobi 
bundles together with the properties we will need afterwards.    

\subsection{Notation and a brief reminder on Jacobi Geometry}
The notions of  Atiyah algebroid of a vector bundle and the associated $Der$-complex are known and are used in many 
other situations. This section is basically meant to fix notation. A more complete introduction to this  can be found in 
\cite{DirJacBun} and its references. Nevertheless, 
the notion of Omni-Lie algebroids was first defined in \cite{CHEN2010799}, in order to study Lie algebroids and local Lie 
algebra structures on vector bundles.

For a vector bundle $E\to M$, we denote its \emph{gauge} or \emph{Atiyah} algebroid by $DE\to M$ and by 
$\sigma\colon DE\to TM$ its anchor. Note that $D$ is a functor from the category of vector bundles with regular, i.e. 
fiberwise invertible, vector bundle morphisms to Lie algebroids. Hence, we denote for a regular $\Phi\colon E\to E'$ by 
	\begin{align*}
	D\Phi\colon DE\to DE'
	\end{align*}
the corresponding Lie algebroid morphism. 
We are mostly dealing with line bundles $L\to M$ for which we have the identity 
$DL=(J^1L)^*\tensor L$, where $J^1L$ is the first jet bundle. The gauge algebroid $DL\to M$ has a (tautological) Lie 
algebroid representation on $L$. The corresponding complex is denoted by 
	\begin{align*}
	\Bigg(\Omega_L^\bullet (M)=\Secinfty(\Anti^\bullet (DL)^*\tensor L), \D_L\bigg).
	\end{align*}
We briefly discuss Jacobi brackets in this setting. A Jacobi bracket is a local Lie algebra structure on the smooth sections 
of a line bundle $L\to M$, i.e. a Lie bracket $\{-,-\}\colon \Secinfty(L)\times \Secinfty(L)\to \Secinfty(L)$, such that
	\begin{align*}
	\{\lambda,-\}\in \Secinfty(DL).
	\end{align*}	 

\begin{remark}
Let $\{-,-\}$ be a Jacobi bracket on a line bundle $L\to M$. Then  there is a unique tensor, called the Jacobi tensor, $J\in 
\Secinfty(\Anti^2(J^1L)^*\otimes L )$, such that 	
	\begin{align*}
	\{\lambda,\mu\}=J(j^1\lambda, j^1 \mu)
	\end{align*}	 
for $\lambda,\mu\in\Secinfty(L)$. Conversely, every $L$-valued $2$-form $J$ on $J^1 L$ defines a skew-symmetric bilinear bracket $\{-,-\}$, but the latter needs not to be a Jacobi bracket. Specifically, it does not need to fulfill the Jacobi identity. However, there is the  notion of a Gerstenhaber-Jacobi bracket 
	\begin{align*}
	[-,-]\colon \Secinfty(\Anti^i(J^1L)^*\otimes L )\times \Secinfty(\Anti^j(J^1L)^*\otimes L )\to
	\Secinfty(\Anti^{i+j-1}(J^1L)^*\otimes L ), 
	\end{align*}
such that the Jacobi identity of $\{-,-\}$ is equivalent to $[J,J]=0$	
see \cite[Chapter 1.3]{2017arXiv170508962T} for a detailed discussion. Finally, a Jacobi tensor defines  a map 
$J^\sharp\colon J^1L\to (J^1L)^*\otimes L=DL$.   
\end{remark}

When $L$ is the trivial line bundle, than the notion of Jacobi bracket boils down to that of \emph{Jacobi pair}.

\begin{remark}[Trivial Line bundle]\label{Rem: TrivLine}
Let $\mathbb{R}_M\to M$ be the trivial line bundle and let 
$J$ be a Jacobi tensor on it. 
Let us denote by $1_M\in \Secinfty(\mathbb{R}_M)$ the canonical global section. 
Using the canonical connection 
	\begin{align*}
	\nabla\colon TM\ni v \mapsto (f\cdot1_M\mapsto v(f)1_M)\in D\mathbb{R}_M,
	\end{align*}
we can see that $DL\cong TM\oplus \mathbb{R}_M$ and hence 
	\begin{align*}
	J^1\mathbb{R}_M=(D\mathbb{R}_M)^*\otimes \mathbb{R}_M= T^*M\oplus \mathbb{R}_M.
	\end{align*}	 
With this splitting, we see that 
	\begin{align*}
	J=\Lambda + \mathbbm{1}\wedge E
	\end{align*}	 	
for some $(\Lambda,E)\in\Secinfty(\Anti^2 TM\oplus TM)$. The Jacobi identity is equivalent to $[\Lambda,\Lambda]+E\wedge
\Lambda=0$ and $\Lie_E\Lambda=0$. The pair $(\Lambda,E)$ is often referred to as \emph{Jacobi pair}.
  Moreover, if we denote by $
\mathbbm{1}^*\in \Secinfty(J^1\mathbb{R}_M)$ the canonical 
section then we can write any 
$\psi\in J^1\mathbb{R}_M$ as $\psi=\alpha+r\mathbbm{1}^*\in \Secinfty(J^1\mathbb{R}_M) $, 
for some $\alpha\in T^*M$ and $r\in \mathbb{R}$. We obtain 
	\begin{align*}
	J^\sharp(\alpha+r\mathbbm{1}^*)=\Lambda^\sharp(\alpha)+r E-\alpha(E)\mathbbm{1}.	
	\end{align*}	
A more detailed discussion about Jacobi structures on trivial line bundles can be found in
 \cite[Chapter 2]{2017arXiv170508962T}. In a similar way, we can see that $\Omega_L(M)^\bullet=\Secinfty(\Anti^\bullet(T^*M\oplus \mathbb{R}_M))=\Secinfty(\Anti^\bullet T^*M \oplus \mathbbm{1}^*\wedge \Anti^{\bullet-1}T^*M)$. Here $\mathbbm{1}^*$ 
 is the canonical section of $\mathbb{R}_M$, moreover the differential $\D_{\mathbb{R}_M}$ is defined by the relations 
 	\begin{align*}
 	\D_{\mathbb{R}_M}(\mathbbm{1}^*)=0 \ \text{ and } \ \D_{\mathbb{R}_M}=\D_{dR}+\mathbbm{1}^*\wedge .
 	\end{align*}
\end{remark}

\subsection{The Omni-Lie Algebroid of a line bundle and its automorphisms}

The omni-Lie algebroid plays the same role as the generalized tangent bundle does in Dirac geometry. In fact, the parallels 
are evidently enormous. Moreover, since the canonical inner product of it will be line-bundle valued, one can easily drop the word \emph{local} Courant algebroid. 
Note that the following definitions and Lemmas are obvious adaptions of the case of $H$-twisted Dirac structure, this is why 
we omit proofs. The non-twisted versions of the following definitions and resulats in Dirac-Jacobi geometry can be found in 
\cite{DirJacBun}.  

\begin{definition}
Let $L\to M$ be a line bundle and let $H\in\Omega_L^3(M)$ be closed. The vector bundle $\mathbb{D}
L:=DL\oplus J^1 L$ together with 
	\begin{enumerate}
	\item the (Dorfman-like, H-twisted) bracket
		\begin{align*}
		[\![(\Delta_1,\psi_1) ,(\Delta_2,\psi_2 )]\!]_H
		=([\Delta_1,\Delta_2],\Lie_{\Delta_1} \psi_2- \iota_{\Delta_2}\D_L\psi_1+\iota_{\Delta_1}\iota_{\Delta_2}H)
		\end{align*}		 
	\item the non-degenerate $L$-valued pairing 
		\begin{align*}
		\bla (\Delta_1,\psi_1) ,(\Delta_2,\psi_2 )\bra := \psi_1(\Delta_2)+\psi_2(\Delta_1)
		\end{align*}
	\item the canonical projection $\pr_D\colon \mathbb{D}L\to DL$ 
	\end{enumerate}	
is called the $H$-twisted Omni-Lie algebroid of $L\to M$.  
\end{definition}

\begin{remark}
If $H=0$, we will refer to $(\mathbb{D}L,[\![- ,-]\!],\bla - ,-\bra)$ as the omni-Lie algebroid.
\end{remark}

We shall now introduce automorphisms of the omni-Lie algebroid, which mirrors the definition of automorphisms of the 
generalized tangent bundle.
\begin{definition}\label{Def: CJ-Aut}
Let $L\to M$ be line bundle and let $H\in \Omega^3_L(M)$ be closed. A pair $(F,\Phi)\in \Aut(\mathbb{D}L)\times \Aut(L)$ 
is called (H-twisted) Courant-Jacobi 
automorphism, if 
	\begin{enumerate}
	\item $D\Phi\colon \pr_D=\pr_D\circ F$
	\item $ \Phi^*\bla-,-\bra=\bla F-,F-\bra$
	\item $F^* [\![-,-]\!]_H=[\![F^*-,F^*- ]\!]_H$
	\end{enumerate}
The group of H-twisted Courant-Jacobi automorphisms is denoted by $\Aut_{CJ}^H(L)$.
\end{definition}

For a line bundle $L\to M$ and $\Phi\in \Aut(L)$, we define 
	 \begin{align*}
	 \mathbb{D}\Phi\colon \mathbb{D}L\ni(\Delta,\alpha)\mapsto (D\Phi(\Delta), (D\Phi^{-1})^*\alpha)\in \mathbb{D}L,
	 \end{align*}
which gives canonically an automorphism $\mathbb{D}\Phi\in \Aut(\mathbb{D}L)$. Moreover, the pair $(\mathbb{D}\Phi,\Phi)$ 
fulfills conditions $i.)$ and $ii.)$ in Definition \ref{Def: CJ-Aut}, nevertheless it is not an (H-twisted) a Courant-Jacobi 
automorphism for an arbitrary $H$. 
For a 2-form $B\in \Omega_L^2(M)$, we define 
	\begin{align*}
	\exp(B)\colon \mathbb{D}L\ni (\Delta,\alpha)\mapsto (\Delta,\alpha+\iota_\Delta B)\in \mathbb{D}L,
	\end{align*}
which also fulfills conditions $i.)$ and $ii.)$ in Definition \ref{Def: CJ-Aut}, seen as pair $(\exp(B),\id)$. 
We can combine this two special kinds of morphisms together with an $H$-dependent action on $\mathbb{D}L$ and find the 
following 

\begin{lemma}\label{Lem: IsoCJ-Aut}
Let $L\to M$ be a line bundle and let $H\in \Omega_L^3(M)$ be closed. If we denote by $Z^2_L(M)$ the closed 2-forms, then
	\begin{align*}
	\mathcal{I}_H\colon Z^2_L(M)\rtimes \Aut(L)\ni (B,\Phi)\mapsto
	 (\exp(B+\iota_\mathbbm{1}(H-\Phi_*H))\circ \mathbb{D}\Phi,\Phi)\in \Aut_{CJ}^H(L)
	\end{align*} 
is an ismorphism of groups. 
\end{lemma}

In a similar way, we can define infinitesimal automorphisms of the Omni lie algebroid
\begin{definition}
Let $L\to M$ be line bundle and let $H\in \Omega^3_L(M)$ be closed. A pair $(D,\Delta)\in \Secinfty(D\mathbb{D}L)\times \Secinfty(DL)$ 
is called infinitesimal (H-twisted) Courant-Jacobi automorphism, if
	\begin{enumerate}
	\item $[\Delta,\pr_D(\varepsilon)]=\pr_D(D(\varepsilon))$
	\item $ \Delta\bla\varepsilon,\chi\bra=\bla D(\varepsilon),\xi\bra + \bla\epsilon,D(\chi)\bra$
	\item $D ([\![\varepsilon,\chi]\!]_H)=[\![D(\epsilon),\chi]\!]_H+[\![\epsilon,D(\chi)]\!]_H$
	\end{enumerate}
for all $\varepsilon,\chi\in \Secinfty(\mathbb{D}L)$.
The lie algebra of infinitesimal (H-twisted) Courant-Jacobi automorphisms is denoted by $\mathfrak{aut} _{CJ}^H(L)$.
\end{definition} 
Note that it is obvious, that the flow of an infinitesimal (H-twisted) Courant-Jacobi autmorphism gives a Courant-Jacobi 
automorphism, in this sense, we can see $\mathfrak{aut}_{CJ}^H(L)$ as the Lie algebra of $\Aut_{CJ}^H(L)$. Similarly to the 
autmorphism case, we have 
\begin{lemma}
Let $L\to M$ be line bundle and let $H\in \Omega^3_L(M)$ be closed. Then 
	\begin{align*}
	\mathfrak{i}_H\colon Z^2_L(M)\rtimes \Secinfty(DL)\ni (B,\Delta)\to ((\Box,\beta)\mapsto ([\Delta,\Box], 
	\Lie_\Delta\beta + \iota_\Box(B-\Lie_\Delta\iota_\mathbbm{1}H)))\in \mathfrak{aut}_{CJ}^H(L)
	\end{align*}
is an isomorphism of Lie algebras.
\end{lemma} 

For every section $(\Delta,\alpha)\in \Secinfty(\mathbb{D}L)$ the map $[\![(\Delta,\alpha),-]\!]_H$ 
is an infinitesimal (H-twisted) 
Courant-Jacobi automorphism, in fact it is realized in $Z^2_L(M)\rtimes \Secinfty(DL)$ by 
	\begin{align*}
	\mathfrak{i}_H(\D_L(\iota_\Delta\iota_\mathbbm{1}H-\alpha),\Delta)=[\![(\Delta,\alpha),-]\!]_H
	\end{align*}
For later use, we want to talk about the flow of infintesimal (H-twisted ) Courant-Jacobi automorphisms and want to compute 
them as explicit as possible. 
\begin{lemma}
Let $L\to M$ be line bundle and let $H\in \Omega^3_L(M)$ be closed. Let additionally $(\alpha,\Delta)\in
Z^2_L(M)\rtimes \Secinfty(DL)$. The flow of $\mathfrak{i}_H(B,\Delta)$ is given by
	\begin{align*}
	\mathcal{I}_H(\gamma_t,\Phi^\Delta_t)&=
	\mathcal{I}_H
	\bigg(-\int_0^t (\Phi_{-\tau}^\Delta)^*B\D \tau, \Phi_t^\Delta\bigg)\\&
	=\exp\big(-\int_0^t(\Phi_{-\tau}^\Delta)^*(B)\D\tau+\iota_\mathbbm{1}(H-(\Phi_t^\Delta)_*H)\big)
	\circ \mathbb{D}\Phi_t^\Delta.
	\end{align*}
\end{lemma}

\begin{corollary}
Let $L\to M$ be line bundle and let $H\in \Omega^3_L(M)$ be closed. For every $(\Delta,\alpha)\in \Secinfty(\mathbb{D}L)$ the 
flow of $[\![(\Delta,\alpha),-]\!]_H$ is given by 
	\begin{align*}
	\exp\big(\int_0^t(\Phi_{-\tau}^\Delta)^*(\D_L\alpha +\iota_{\Delta}H)\D\tau\big)
	\circ \mathbb{D}\Phi_t^\Delta.
	\end{align*}
\end{corollary}

\subsection{Dirac-Jacobi bundles}
After having discussed the arena, we want to introduce the subbundles of interest: so-called \emph{Dirac-Jacobi Bundles}. As the name suggest, they are the analogue of Dirac structures on the generalized tangent bundle. In fact, the definition is 
(up to some obvious replacements) the same.

\begin{definition}
Let $L\to M$ be a line bundle and $H\in\Omega_L^3(M)$ . A subbundle $\mathcal{L}\subseteq \mathbb{D}L$ is called a ($H$-twisted) Dirac-Jacobi structure, if 
	\begin{enumerate}
	\item $\mathcal L$ is involutive with respect to $[\![- ,-]\!]_H$,
	\item $\mathcal L$ is maximally isotropic with respect to $\bla - ,-\bra$ .
	\end{enumerate}
Moreover, if $H=0$, we will call $\mathcal{L}$ simply Dirac-Jacobi structure.
\end{definition}

\begin{example}\label{Ex: Jacobi}
Let $L\to M$ be a line bundle and let $J\in \Secinfty(\Anti^2 (J^1L)^*\tensor L)$ be a Jacobi structure, then
	\begin{align*}
	\mathcal{L}_J:=\{(J^\sharp(\psi),\psi)\in \mathbb{D}L\ | \ \psi\in J^1 L\}
	\end{align*}
is a Dirac-Jacobi structure.
\end{example}

\begin{proposition}\label{Prop: DJtoJ}
Let $L\to M $ be a line bundle and let $\mathcal{L}\subseteq \mathbb{D}L$ be a Dirac-Jacobi bundle, such that 
	\begin{align*}
	DL\cap \mathcal{L}=\{0\}.
	\end{align*}
Then there is a unique Jacobi structure $J\in \Secinfty(\Anti^2 (J^1L)^*\tensor L)$, such that $\mathcal{L}_J=\mathcal{L}$
\end{proposition}

\begin{proof}
The result follows the same lines as the well-known fact in Poisson geometry.
\end{proof}
Another interesting example of Dirac-Jacobi bundles, which also plops up in Jacobi geometry, is 
\begin{definition}
Let $L\to M$ be a line bundle. A Dirac-Jacobi structure $\mathcal{L}\subseteq \mathbb{D}L$ is called 
of homogeneous Poisson type, if
	\begin{align*}
	\rank(\mathcal{L}\cap DL)=1.
	\end{align*}	 
\end{definition} 
The name of these objects is justified by the following 

\begin{lemma}\label{Lem: HomPoi}
Let $L\to M$ be a line bundle and let $\mathcal{L}\subseteq\mathbb{D}L$ a Dirac-Jacobi structure of homogeneous Poisson type, then for every point $p\in M$ there exists a local trivialization $L_U=U\times \mathbb{R}$, a flat connection 
$\nabla\colon TU\to DL_U\cong TU\oplus\mathbb{R}_U$ and a homogeneous Poisson structure $\pi\in\Secinfty(\Anti^2TU)$ with homogeneity $Z\in\Secinfty(TM)$, such that 
	\begin{align*}
	\mathcal{L}\at{U}=
	\{(r(\mathbbm{1}-\nabla_Z)+\nabla_{\pi^\sharp(\alpha)},\alpha +\alpha(Z)\mathbbm{1}^*)
	\in \mathbbm{D}L\at{U} \ | \ h\in \mathbb{R}, \  \alpha\in T^*M\},
	\end{align*}
where we use the inclusion $T^*M\to J^1L$ by $\alpha(\nabla_X)=\alpha(X)$ and $\alpha(\mathbbm{1})=0$.
\end{lemma}

\begin{proof}
Let $p\in M$ and $U\subseteq M$ be an open subset containing $p$, such that $L_U\cong U\times \mathbb{R}$ with corresponding 
trivialization of the gauge algebroid $DL_U=TU\oplus \mathbb{R}_U$, and hence we are using the canonical flat connection 
$\nabla^{\mathrm{can}}\colon TU\to TU\oplus \mathbb{R}_U$. In a possibly 
smaller neighbourhood, notated also by $U$, we find a non-vanishing section $\Delta=(-X,f)\in \Secinfty(\mathcal{L}\cap DL)$.
We can distinguish two cases: the first is that $f(p)\neq 0$, the we find a (possibly smaller) neighbourhood of $p$, such 
that $f$ is non-vanishing, hence $(-\frac{X}{f},1)=:(-Z,1)$ spans  $\mathcal{L}\cap DL$ in that neighbourhood. Exploting the 
isotropy, we see that $\mathcal{L}\at{U}$ is of the form 
	\begin{align*}
	\{(h\mathbbm{1}+ \nabla^\mathrm{can}_Y,\alpha+\alpha(Z)\mathbbm{1}^*)\in \mathbb{D}L_U\ | \ 
	h\in \mathbb{R}, \ \alpha\in T^*U\}
	\end{align*}
and not further specified $Y\in TU$, since the $J^1L_U$ part has to vanish at sections of the form 
$r(\mathbbm{1}-\nabla^\mathrm{can}_Z)$. We can write this as 
	\begin{align*}
	\{(h(\mathbbm{1}-\nabla^\mathrm{can}_Z)+ \nabla^\mathrm{can}_{hZ+Y},\alpha+\alpha(Z)\mathbbm{1}^*)\in \mathbb{D}L_U\ | \ 
	h\in \mathbb{R}, \ \alpha\in T^*U\}.
	\end{align*}
Note that, because of the isotropy, $hZ+Y$ is completely determined by $\alpha$, hence there is a bi-vector $\pi\in \Secinfty(\Anti^2 TU)$ such that $\pi^\sharp(\alpha)=hZ+Y$ and we can write
	\begin{align*}
	\mathcal{L}\at{U}=
	\{(h(\mathbbm{1}-\nabla^\mathrm{can}_Z)+ \nabla^\mathrm{can}_{\pi^\sharp(\alpha)},
	\alpha+\alpha(Z)\mathbbm{1}^*)\in \mathbb{D}L_U\ | \ 
	h\in \mathbb{R}, \ \alpha\in T^*U\}.
	\end{align*}
The claim follows by using the flatness of $\nabla^\mathrm{can}$ and the involutivity of $\mathcal{L}$.

Now we have to treat the case $f(p)=0$. Since $\Delta=(-X,f)$ is non-vanishing, we conclude that $X(p)\neq 0$, hence there is a closed two form $\beta\in \Secinfty(T^* U)$ such that $\beta(X)=-1$ around $p$. We define the flat connection 
	\begin{align*}
	\nabla\colon TU\ni Y\mapsto \nabla^\mathrm{can}_Y +\beta(Y)\mathbbm{1} \in DL_U.
	\end{align*}
With this connection we see that $\Delta=(f-1)\mathbbm{1}-\nabla_X$ and since $f(p)=0$, we have that $f-1\neq 0$ in a whole neighbourhood of $p$ and hence we choose $\Delta'=\frac{1}{f-1}\Delta$ as a generating section of $\mathcal{L}\cap DL$ around 
$p$. We can now repeat the same argument as for the case $f(p)\neq 0$ by using the connection $\nabla$ instead of $\nabla^\mathrm{can}$, since $\Delta'= \mathbbm{1}-\nabla_Z$ for $Z=\frac{1}{f-1} X$.
\end{proof}

In the category of Dirac-Jacobi bundles there are not just automorphism of the omni-Lie algebroid as morphisms, one of the possibilities is to include so-called \emph{backwards transformations} as in the Dirac geometry case.  

\begin{definition}
Let $L_i\to M_i$ for $i=1,2$ be two line bundles and let $\Phi\colon L_1\to L_2$ be a regular line bundle morphism
covering $\phi\colon M_1\to M_2$. Let $\mathcal{L}\subseteq\mathbb {D}L_2$ be a Dirac-Jacobi bundle. The bundle 
	\begin{align*}
	\mathfrak{B}_{\Phi}(\mathfrak{L}):= \{(\Delta_p, (D\Phi)^*\alpha_{\phi(p)})\in\mathbb {D}L_1\ | \ 
	(D\Phi(\Delta_p),\alpha_{\phi(p)})\in\mathfrak{L}\}
	\end{align*}	
is called Backwards transformation of $\mathcal{L}$.   
\end{definition}

The backwards transform of a Dirac-Jacobi bundle need not to be Dirac-Jacobi anymore, 
but there are sufficient conditions on the 
subbundle $\mathcal{L}$ and the line bundle morphism $\Phi$ which can be seen, i.e. in \cite{DirJacBun}:

\begin{theorem}\label{Thm: CleanInt}
Let $\Phi\colon L_1\to L_2$ be a regular line bundle morphism over $\phi:M_1\to M_2$ and let
$\mathcal{L}\in \mathbb{D}L_2 $ be a Dirac-Jacobi bundle. If $\ker D\Phi^*\cap \phi^*\mathcal{L}$ has constant rank, then $\mathfrak{B}_\Phi(\mathcal{L})$ is a Dirac-Jacobi bundle. 
\end{theorem}

\begin{proof}
The proof can be found in \cite[Proposition 8.4]{DirJacBun}. 
\end{proof}

\begin{remark}
Note that for a line bundle automorphism $\Phi\in\Aut{L}$, we have that 
$\mathbb{D}\Phi(\mathfrak{L})=\mathfrak{B}_{\Phi^{-1}}(\mathfrak{L})$. but not every backwards transform needs to be of this 
form.
\end{remark}

\section{Submanifolds and Euler-like Vector Fields}
In this subsection we want to discuss Euler-like vector fields. These vector fields, in particular, induce a homogeneity structure on the manifold, which is equivalent, under some additional conditions which are in our case always fulfilled, that the manifold is total space of a vector bundle, see e.g. \cite{GR2012}. This total space turns out to be the normal bundle for some submanifold, which is an input datum for an Euler-like vector field. Nevertheless, we will not go more in details with these features, since we work directly with tubular neighbourhoods. 
We will begin collecting facts about tubular neighbourhoods, submanifolds and corresponding mappings and describe afterwards 
the notion of Euler-like vector fields and extend this notion the derivations of a line bundle.

\subsection{Normal Bundles and tubular neighbourhoods}

For pair of manifolds $(M,N)$, i.e. a submanifold $N\hookrightarrow M$, we denote
	\begin{align*}
	\nu(M,N)=\frac{TM\at{N}}{TN}
	\end{align*}
the normal bundle. If the ambient space is clear, we will just write $\nu_N$ instead. Given a map of pairs
	\begin{align*}
	\Phi\colon(M,N)\to (M',N'),
	\end{align*}
i.e. a map $\Phi\colon M\to M$, such that $\Phi(N)\subseteq N'$, we denote by 
	\begin{align*}
	\nu(\Phi)\colon \nu(M,N)\to \nu(M',N')
	\end{align*}
the induced map on the normal bundle. For a vector field $X$ on $M$ tangent to $N$, we have that the flow $\Phi^X_t$ is a map of pairs from $(M,N)$ to itself. Hence we define
	\begin{align*}
	T\nu(X)=\frac{\D}{\D t}\At{t=0} \nu(\Phi^X_t)\in \Secinfty(T\nu_N).
	\end{align*}	  
Moreover, for a vector bundle $E\to M$ and $\sigma \in \Secinfty(E)$, such that $\sigma \at{N}=0$ for a submanifold $N\hookrightarrow M$, we denote by
	\begin{align*}
	\D^N\sigma \colon \nu_N\to E\at{N}
	\end{align*}
the map which is $\nu(\sigma)$, for $\sigma$ seen as a map $\sigma\colon (M,N)\to (E,M)$, followed by the canonical identification $\nu(E,M)=E$, given by 
	\begin{align*}
	C_E\colon E\ni v_p\to [\frac{\D}{\D t}\At{t=0} tv_p]_{TM}\in  \nu(E,M).
	\end{align*}
Before we prove the next results, we want to find a useful description of $C_E^{-1}$. Let us therefore consider a curve $\gamma\colon I\to E$ for an open interval $I$ containing $0$, such that $\gamma(0)=0_p$ for $p\in M$, then one can prove in local coordinates 
	\begin{align}\label{Eq: CanIso}
	C_E^{-1}([\frac{\D}{\D t}\At{t=0}\gamma(t)])=\lim_{t\to 0} \frac{\gamma(t)}{t}.
	\end{align}

\begin{proposition}\label{Prop: Identification}
Let $E_i\to M_i$ vector bundles for $i=1,2$ and let $\Phi\colon E_1\to E_2$ be a vector bundle morphism. Then, for 
$\Phi\colon (E_1,M_1)\to (E_2,M_2)$, 
	\begin{align*}
	C^{-1}_{E_2}\circ\nu(\Phi)\circ C_{E_1}=\Phi
	\end{align*}
\end{proposition}

\begin{proof}
Let $v_p\in E_1$, then 
	\begin{align*}
	(C^{-1}_{E_2}\circ\nu(\Phi)\circ C_{E_1})(v_p)&
	=(C^{-1}_{E_2}\circ\nu(\Phi))([\frac{\D}{\D t}\At{t=0} tv_p]_{TM_1})\\&
	=C^{-1}_{E_2}([T\Phi\frac{\D}{\D t}\At{t=0} tv_p]_{TM_2})\\&
	=C^{-1}_{E_2}([\frac{\D}{\D t}\At{t=0} t\Phi(v_p)]_{TM_2})\\&
	=\Phi(v_p)
	\end{align*}
\end{proof}

\begin{proposition}\label{Prop: NormalDer}
Let $E_i\to M$ be  vector bundles for $i=1,2$ and let $\Phi\colon E_1\to E_2$ be a vector bundle morphism covering the 
identity. Then, for every section $\sigma\in \Secinfty (E_1)$, such that $\sigma\at{N}=0$ for some submanifold 
$N\hookrightarrow M$, 
	\begin{align*}
	\D^N\Phi(\sigma)=\Phi(\D^N\sigma)
	\end{align*}
holds.
\end{proposition}

\begin{proof}
We consider the map $\Phi(\sigma)\colon (M,N)\to (E_2,M)$, then we have 
	\begin{align*}
	C_{E_2}^{-1}\circ \nu(\Phi(\sigma))&
	=C_{E_2}^{-1}\circ\nu(\Phi)\circ \nu(\sigma)\\&
	=C_{E_2}^{-1}\circ \nu(\Phi)\circ C_{E_1}\circ C_{E_1}^{-1}\circ \nu(\sigma)\\&
	=\Phi \circ  C_{E_1}^{-1}\circ \nu(\sigma)
	\end{align*}
and the claim follows if we restrict this maps.
\end{proof}

\begin{proposition}\label{Prop: Endomorphism}
Let $(M,N)$ be a pair of manifolds and let $X\in\Secinfty(TM)$, such that $X\at{N}=0$. Then 
	\begin{align*}
	T\Phi_t^X\at{N}=\exp(tD_X)
	\end{align*}	 
for a unique $D_X\in \Secinfty(\End(TM\at{N}))$, moreover $TN \subseteq \ker(D_X)$ and
\begin{center}
		\begin{tikzcd}
 		TM\at{N} \arrow[r, "D_X"]\arrow[d]&   TM\at{N}    \\
 		\nu_N \arrow[ru,swap,  "\D_N X"] 
		\end{tikzcd}
	\end{center}
commutes. 
\end{proposition}

\begin{definition}
Let $(M,N)$ be a pair of manifolds. A tubular neighbourhood of $N$ is an open subset $U\subseteq M$ 
containing $N$ together with a 
diffeomorphism 
	\begin{align*}
	\psi\colon \nu_N\to U,
	\end{align*}
such that $\psi\at{N}\colon N\to N$ is the identity and for $\psi\colon (\nu_N,N)\to (M,N)$ the map 
	\begin{align*}
	\nu(\psi)\colon \nu(\nu_N,N)\to \nu_N
	\end{align*}
is inverse of   $C_{\nu_N}\colon \nu_N\to \nu(\nu_N,N)$.
\end{definition}

\subsection{Euler-like Vector fields and Derivations}
In this part, we recall basically just the notion of Euler-like vector fields from \cite{2016arXiv160505386B} and extend 
this notion to derivations of a line bundle. 
 
\begin{definition}
Let $(M,N)$ be a pair of manifolds. A vector field $X\in\Secinfty(TM)$ is called Euler-like, if 
	\begin{enumerate}
	\item $X\at{N}=0$,
	\item $X$ has complete flow,
	\item $T\nu(X)=\mathcal{E}$,
	\end{enumerate}
where $\mathcal{E}$ is the Euler vector field on $\nu_N\to N$. 	
\end{definition}

\begin{proposition}
Let $(M,N)$ be a pair of manifolds, then there exists an Euler-like vector field.
\end{proposition}

\begin{proof}
Let us choose a tubular neighbourhood  
	\begin{align*}
	\psi\colon \nu_N\to U.
	\end{align*}
For the vector field $X=\psi_* \mathcal{E}$ multiplied by a suitable bump function which is 1 in a neighbourhood of $N$, we 
have 
	\begin{align*}
	T\nu(X)&=\frac{\D}{\D t}\At{t=0}\nu(\Phi^X_t)= \frac{\D}{\D t}\At{t=0}\nu(\psi\circ\Phi^\mathcal{E}_t\circ \psi^{-1})\\&
	=\frac{\D}{\D t}\At{t=0}\nu(\psi)\circ\nu(\Phi^\mathcal{E}_t)\circ \nu( \psi^{-1})\\&
	=\frac{\D}{\D t}\At{t=0} \Phi^\mathcal{E}_t=\mathcal{E},
	\end{align*}
where we used Proposition \ref{Prop: Identification} and the fact that $\nu( \psi)=C^{-1}_{\nu_N}$.
\end{proof}

\begin{lemma}\label{Lem: ExTubNei}
Let $M$ be a manifold, $N\hookrightarrow M$ a submanifold and $X\in \Secinfty(TM)$ be a Euler-like vector field. Then there 
exists a tubular unique neighbourhood embedding 
	\begin{align*}
	\psi\colon \nu_N\to U,
	\end{align*}
 such that $\psi^* X=\mathcal{E}$.
\end{lemma}

\begin{proof}
The proof can be found in \cite{2016arXiv160505386B}.
\end{proof}


\begin{proposition}\label{Prop: NormDertoEuler}
Let $(M,N)$ be a pair of manifolds and let $X\in \Secinfty(TM)$ be a vector field, such that $X\at{N}=0$ and is complete. 
Then
$X$ is Euler-like, if and only if $\D^N X$ followed by the projection $TM\at{N}\to \nu_N$ is identity.
\end{proposition}

\begin{proof}
Let $X\in \Secinfty(TM)$ be given as in the proposition. According to Proposition \ref{Prop: Endomorphism}, there exists a 
unique $D_X\in \Secinfty(\End(TM\at{N}))$, such that $T\Phi_t^X\at{N}=\exp(tD_X)$. Let $[X_p]\in \nu_N$, then 
	\begin{align*}
	\nu(\Phi^x_t)([X_p])&
	=[T\Phi^X_t(X_p)]=[\exp(tD_X)(X_p)].
	\end{align*}
This is just equal to the flow of the Euler vector field, if $\pr_{\nu_N}\circ D_X=\id_{\nu_N}$. Using Proposition \ref{Prop: 
Endomorphism}, we have $\D^N X=D_X$ and hence the claim.
\end{proof}

Note that for a pair of manifolds $(M,N)$ and a Euler like vector field $X\in \Secinfty(TM)$, the set 
	\begin{align*}
	\{p\in M\ | \ \lim_{t\to -\infty} \Phi_t^X(p) \text{ exists and lies in }  N\}
	\end{align*}
is an open subset in $M$ containing $N$, such that that the action of  $\Phi_t^X$ shrinks to this set. Moreover, for 
a tubular neighbourhood $\psi\colon \nu_N \to U$, such that $\psi^*X=\mathcal{E}$, we have that 
	\begin{align*}
	U=\{p\in M\ | \ \lim_{t\to -\infty} \Phi_t^X(p) \text{ exists and lies in }  N\}.
	\end{align*}
Let us denote by $\lambda_s=\Phi^X_{\log(s)}\at{U}$. 
We obtain,  that $\lambda_s$ is smooth for all $s\in \mathbb{R}_0^+$. Moreover, we have that 
	\begin{align}\label{Eq: FlowCom}
	\psi\circ \lambda_s=\kappa_s\circ \psi,
	\end{align}	 
where we denote by $\kappa_s\colon \nu_N\to \nu_N$ the map $[X_p]\mapsto [s X_p]$. Note that $\kappa_0\colon \nu_N\to N$ coincides with the bundle projection, to be more precise $k_0=\pr_\nu\circ j$, where $\pr_\nu$ is the bundle projection and 
$j\colon N\to \nu_N$ the canonical inclusion. 

Let us add now the line bundle case 
\begin{definition}
Let $L\to M$ be a line bundle and $N\hookrightarrow M$ be a submanifold. A derivation $\Delta\in \Secinfty(DL)$ is called 
Euler-like, if 
	\begin{enumerate}
	\item $\Delta\at{N}=0$,
	\item $\sigma(\Delta)$ is an Euler-like vector field.
	\end{enumerate}
\end{definition}
This definition turns out to be the correct one for our purposes, since we can prove basically all results, which are available for Euler-like vector fields. Let us start collecting them.
\begin{proposition}
Let $L\to M$ be a line bundle and let $\Delta\in \Secinfty(DL)$ be an Euler-like derivation with respect to 
$N\hookrightarrow M$, then the flow $\Phi^\Delta_t\in \Aut(L)$ of $\Delta$ induces the map
	\begin{align*}
	\Lambda_s=\Phi_{\log(s)}^\Delta, 
	\end{align*}
which can be, restricted to $U=\{p\in M\ | \ \lim_{t\to -\infty} \Phi_t^{\sigma(X)}(p) \text{ exists and lies in }  N\}$, 
extended smoothly to $s=0$. Moreover, the map
	\begin{align*}
	\Lambda_0\colon L\to L_N
	\end{align*}
is a regular line bundle morphism. 
\end{proposition}

\begin{proof}
The proof is an easy verification using a tubular neighbourhood $\psi\colon \nu_N\to U$, such that 
$\psi^*\sigma(X)=\mathcal{E}$.
\end{proof}

\begin{definition}
Let $L\to M$ be a line bundle and $N\hookrightarrow M$ be a submanifold. A fat tubular neighbourhood is a regular line bundle  morphism 
	\begin{align*}
	\Psi\colon L_\nu \to L_U,
	\end{align*}
where the line bundle $L_\nu$ is given by the pull-back
	\begin{center}
		\begin{tikzcd}
 		L_\nu \arrow[r]\arrow[d]&   L_N \arrow[d]   \\
 		\nu_N\arrow[r] & N
		\end{tikzcd}
	,
	\end{center}	
covering a tubular neighbourhood $\psi\colon \nu_N\to U$, such that $\Psi\at{N}\colon L_N\to L_N$ is the identity. 
\end{definition}

\begin{lemma}\label{Ex: ExFatTubNei}
Let $L\to M$ be a line bundle, let $N\hookrightarrow M$ be a submanifold and let $\psi\colon\nu_N\to U$ be a tubular 
neighbourhood. Then there exists a fat tubular neighbourhood covering $\psi$. 
\end{lemma}

\begin{proof}
The proof can be found in \cite[Chapter 3]{2017arXiv170508962T}. 
\end{proof}

For a line bundle $L\to N$ and a vector bundle $E\to N$ there is always a canonical 
Derivation $\Delta_{\mathcal{E}}\in \Secinfty(DL_E)$, such that $\sigma(\Delta_{\mathcal{E}})=\mathcal{E}$ constructed as 	
follows:
Consider the map 
	
\begin{center}
		\begin{tikzcd}
 		L_E \arrow[r, "P"]\arrow[d]&   L \arrow[d]   \\
 		E\arrow[r, "p"] & N
		\end{tikzcd}
	\end{center}
and the corresponding map $DP\colon L_E\to L_N$. We have that canonically $\ker(DP)\cong \mathrm{Ver}(E)$, which 
induces a flat (partial) connection $\nabla\colon \mathrm{Ver}(E)\to DL_\nu$. Since the Euler vector field is 
canonically vertical, we can define $\Delta_\mathcal{E}=\nabla_\mathcal{E}$.  

\begin{proposition}
Let $L\to N$ be a line bundle and let $E\to N$ be a vector bundle. Then the flow $\Phi_t$ of 
$\Delta_\mathcal{E}\in \Secinfty(DL_E)$
is given by 
	\begin{align*}
	\Phi_t(v_p,l_p)=(\E^t\cdot v_p,l_p)
	\end{align*}
for all $(v_p,l_p)\in L_E$.
\end{proposition}

\begin{proof}
This proof is an easy verification using the fact that $\Phi_t$ covers the flow of the Euler vector field.
\end{proof}

Note that for the flow $\Phi_t$ of the canonical Euler-like derivation $\Delta_\mathcal{E}\in \Secinfty(DL_E)$, we have that 
	\begin{align*}
	P_s=\Phi_{\log(s)}\colon L_E\to L_E
	\end{align*}
is defined for all $s>0$ and can be extended smoothly to $s=0$, moreover $P_0$ coincides with the canonical projection
$P\colon L_E\to L$ followed by the canonical inclusion $J\colon L\to L_E$.

\begin{lemma}\label{Ex: UnFatTubNei}
Let $L\to M$ be a line bundle, let $N\hookrightarrow M$ be a submanifold and let $\Delta\in \Secinfty(DL)$ be an Euler-like 
derivation. Then there is a unique fat tubular neighbourhood  $\Psi\colon L_\nu\to L_U$, such that $\Psi^*\Delta=
\Delta_{\mathcal{E}}$.  
\end{lemma}

\begin{proof}
First, we want to proof existence. It is clear that any such $\Psi$ has to cover the unique tubular neighbourhood
$\psi\colon\nu_N\to U$, such that $\psi^*\sigma (\Delta)=\mathcal{E}$. So let us choose a fat tubular neighbourhood
$\tilde\Psi\colon L_\nu\to L_U$ covering $\psi$. We consider now $\tilde\Psi^*\Delta\in \Secinfty(DL_\nu)$. 
We have $\sigma(\tilde\Psi^*\Delta)=\psi^*\sigma(\Delta)=\mathcal{E}$. Hence $\sigma(\Delta_{\mathcal{E}})= \sigma(\tilde\Psi^*\Delta)$. 
Consider now the derivation $\Box=\Delta_\mathcal{E}-\tilde\Psi^*\Delta$ and 	
	\begin{align*}
	\Box_t=-\frac{1}{t}\Phi_{\log(t)}^*\Box,
	\end{align*}	
where $\Phi_t $ is the flow of $\Delta_\mathcal{E}$.
Let us denote the flow of $\Box_t$ by $\phi_t$. Note that it is complete, since $\sigma(\Box_t)=0$, indeed there is even a explicit formula for it, which we do not use. Note however, that $\phi_t\in \Gau(L_\nu)$ for all $t\in \mathbb{R}$. 
Let us compute
	\begin{align*}
	\frac{\D}{\D t}\phi_t^*(\Delta_{\mathcal{E}}+t\Box_t)&
	=\phi^*_t([\Box_t,\Delta_\mathcal{E}] +\frac{\D}{\D t}t\Box_t)\\&
	=\phi^*_t([\Box_t,\Delta_\mathcal{E}] -\frac{\D}{\D t}\Phi^*_{\log(t)}\Box)\\&
	=\phi^*_t([\Box_t,\Delta_\mathcal{E}] -\frac{1}{t} [\Delta_{\mathcal{E}},\Phi^*_{\log(t)}\Box])\\&
	=\phi^*_t([\Box_t,\Delta_\mathcal{E}] + [\Delta_{\mathcal{E}},\Box_t])\\&
	=0.
	\end{align*}	  
Hence we see $\Delta_\mathcal{E}=\phi_0^*(\Delta_\mathcal{E})=\phi_1^*(\Delta_\mathcal{E}+\Box_1)=
\phi_1^*(\tilde\Psi^*\Delta)$.
Therefore, we have that the map $\Psi=\tilde\Psi\circ\phi_1$ will do the job, since obviously $\phi_1\at{N}=\id$. 

Let us now assume that we have 
$\Psi_1,\Psi_2\colon L_\nu\to L_U$, such that $\Psi_1^*\Delta=\Psi_2^*\Delta=\Delta_\mathcal{E}$. Note that since both have 
to cover the unique $\psi\colon \nu_N\to U$, the target $L_U$ is for both the same. Let us consider 
$\Xi:=\Psi_1^{-1}\circ\Psi_2 \colon L_\nu\to L_\nu$, 
which covers the identity, which implies that there is a nowhere vanishing 
function $f\in \Cinfty(\nu_N)$, such that $\Xi(l_p)=f(p)l_p$ for all $l_p\in L_\nu$. Moreover, we have that $\Xi\at{N}=
\id_{L_\nu}\at{N}$, hence $f(0_n)=1$ for all $n\in N$, 
and  $\Xi^*\Delta_{\mathcal{E}}=\Delta_\mathcal{E}$. We consider now an arbitrary section $\lambda\in \Secinfty(L_\nu)$ and 
compute
	\begin{align*}
	\Delta_\mathcal{E}(\lambda)&=(\Xi^*\Delta_\mathcal{E})(\lambda)\\&
	=\Xi^* (\Delta_\mathcal{E}(\Xi_*\lambda))\\&
	=\frac{1}{f}(\Delta_\mathcal{E}(f\lambda))\\&
	=\frac{\mathcal{E}(f)}{f}\lambda + \Delta_\mathcal{E}(\lambda).
	\end{align*}	 
Hence $\mathcal{E}(f)=0$, which means that $f=\pr_\nu^*g$ for some function $g\in \Cinfty(N)$, 
but since $1=f(0_n)=g(n)$ for all $n\in N$, we have that $\Xi=\id_{L_\nu}$.
\end{proof}

For a line bundle $L\to M$, a submanifold $N$ and an Euler-like derivation $\Delta\in \Secinfty(DL)$, we have that 
	\begin{align*}
	\Lambda_s:=\Phi^\Delta_{\log(s)}\colon L_U\to L_U
	\end{align*}
is well defined for $s>0$ and can be extended smoothly to $s=0$, where $L_U$ is the target of the unique fat tubular 
neighbourhood $\Psi \colon L_\nu\to L_U$, such that $\Psi^*\Delta=\Delta_\mathcal{E}$. Moreover, we have that 
	\begin{align}\label{Eq: CommutationD}
	\Lambda_s\circ \Psi=\Psi\circ P_s
	\end{align}	 
for all $s\geq 0$. Note that if we project this equation to the manifold level, this simply gives Eq. \ref{Eq: FlowCom}.

\section{Normal Forms of Dirac-Jacobi bundles}\label{Sec: NormFroDJ}
Using the techniques of Euler-like derivations, we want to prove a normal form theorem for Dirac- Jacobi bundles. In fact, if the submanifold $N$ is a transversal, then we can find special Euler like derivations 
which are, in some sense, controlling the 
behaviour of the Dirac-Jacobi bundles near $N$. The aim is now to prove the existence of this special kind of Euler-like derivations and afterwards, we are able to prove a normal form theorem. and conclude some corolloraries from it.

\begin{definition}
Let $L\to M$ be a line bundle, let $H\in \Omega_L^3(M)$ be closed  and let  $\mathcal{L}\subseteq \mathbb{D}L$ be a
$H$-twisted Dirac-Jacobi bundle. A submanifold $N\hookrightarrow M$ is called transversal, if 
	\begin{align*}
	DL_N+\pr_D \mathcal{L}\at{N}=(DL)\at{N}.
	\end{align*}
\end{definition}

\begin{proposition}
Let $L\to M$ be a line bundle, let $H\in \Omega_L^3(M)$ be closed, let  $\mathcal{L}\subseteq \mathbb{D}L$ be a
$H$-twisted Dirac-Jacobi bundle and let $N\hookrightarrow M$ be a transversal. Then 
	\begin{align*}
	\mathfrak{B}_{I}(\mathfrak{L}):= \{(\Delta_p, (DI)^*\alpha_{\iota(p)})\in\mathbb {D}L_1\ | \ 
	(DI(\Delta_p),\alpha_{\phi(p)})\in\mathcal{L}\}
	\end{align*}
is a $I^*H$-twisted Dirac-Jacobi bundle, where $I\colon L_N\to L$ is the canonical inclusion.
\end{proposition}

\begin{proof}
This is an easy consequence of Theorem \ref{Thm: CleanInt}.
\end{proof}

\begin{lemma}\label{Lem: Iso:Pullback-Backwards}
Let $L\to M$ be a line bundle, let $H\in \Omega_L^3(M)$ be closed, let  $\mathcal{L}\subseteq \mathbb{D}L$ be a
$H$-twisted Dirac-Jacobi bundle and let $\iota\colon N\hookrightarrow M$ be a transversal. The backwards transformation 
$\mathfrak{B}_{I}(\mathfrak{L})$ is canonically isomorphic (as vector bundles) to the fibered product 
	\begin{center}
		\begin{tikzcd}
 		I^! \mathcal{L} \arrow[r]\arrow[d]&   \mathcal{L}\arrow[d, "\pr_D"]   \\
 		DL_N \arrow[r, "DI"] & DL
		\end{tikzcd}
		.
	\end{center}	
\end{lemma}

\begin{proof}
We consider the linear map 
	\begin{align*}
	\Xi\colon I^! \mathcal{L}_p\ni (\Delta_p,(\Box_{\iota(p)},\alpha_{\iota(p)}))\mapsto
	(\Delta_p,DI^*\alpha_{\iota(p)})\in \mathfrak{B}_{I}(\mathfrak{L}),
	\end{align*}
which is well-defined since $DI(\Delta_p)=\Box_{\iota(p)}$. We claim now that this map is injective, let us therefore consider 
$(\Delta_p,(\Box_{\iota(p)},\alpha_{\iota(p)}))\in\ker(\Xi)$. It follows immediately, that $\Delta_p=0$ and hence $\Box_{\iota(p)}=0$. If 
$(0,\alpha_{\iota(p)})\in \mathcal{L}$ then $\alpha_{\iota_{p}}\in \mathrm{Ann}(\pr_D L )$. Since $DI^*\alpha_{\iota(p)}0 =0$, we have that $\alpha_{\iota(p)}\in \mathrm{Ann}(DL_N)$, hence $\alpha_{\iota(p)}=0$ and the claim follows. 

For dimenional reasons we have that $\Xi$ is an isomorphism. 
\end{proof}

\begin{proposition}\label{Prop: ExEulerlikeDer}
Let $L\to M$ be a line bundle, let $H\in \Omega_L^3(M)$ be closed, let  $\mathcal{L}\subseteq \mathbb{D}L$ be a
$H$-twisted Dirac-Jacobi bundle and let $N\hookrightarrow M$ be a transversal. Then there exists $\varepsilon\in 
\Secinfty(\mathcal{L})$, such that $\varepsilon\at{N}=0$ and $\pr_D(\varepsilon)$ is Euler-like. 
\end{proposition}

\begin{proof}
We consider the exact sequence 
	\begin{align*}
	0\to\mathfrak{B}_{I}(\mathfrak{L})\to \mathcal{L}\at{N}\to \nu_N\to 0,  
	\end{align*}
where the first arrow is defined by 
the identifiaction $\mathfrak{B}_I(\mathcal{L})\cong I^!\mathcal{L}$ from Lemma \ref{Lem: Iso:Pullback-Backwards} followed by 
the canonical map $I^!\mathcal{L}\to \mathcal{L}$. The second arrow is the projection $\pr_D\colon \mathcal{L}\at{N}\to DL
\at{N}$ followed by the symbol map $
\sigma\colon DL\at{N}\to TM\at{N}$and finally followed by the the projection to the normal bundle 
$\pr_{\nu_N}\colon TM\at{N}\to \nu_N$. Let us choose a section 
$\varepsilon\in \Secinfty(\mathcal{L})$ with $\varepsilon\at{N}=0$, such that $\D^N\varepsilon\colon \nu_N\to 
\mathcal{L}\at{N}$ defines a splitting of the sequence. We consider now 
	\begin{center}
		\begin{tikzcd}
 		0\arrow[r] & I^! \mathcal{L} \arrow[r]\arrow[d]&   \mathcal{L}\at{N}\arrow[r]\arrow[d]&\nu_N 
 		\arrow[r]\arrow[d] & 0   \\
 		0\arrow[r] & TN \arrow[r] &   TM\at{N}\arrow[r]&\nu_N\arrow[r]& 0
		\end{tikzcd}
	\end{center}	
and see that if $\D^N\varepsilon$ splits the above sequence then $(\sigma\circ \pr_D)\D^N\epsilon$ splits the lower sequence. 
Using Proposition \ref{Prop: NormalDer}, we see that $(\sigma\circ \pr_D)\D^N\epsilon=\D^N((\sigma\circ \pr_D)(\varepsilon))$ 
and by Proposition \ref{Prop: NormDertoEuler}, we see that $T\nu(\sigma\circ \pr_D)(\varepsilon)=\mathcal{E}$. Multiplying 
$\varepsilon$ by a suitable bump function we may arrange that $(\sigma\circ \pr_D)(\varepsilon)$ is complete and hence  
an Euler-like vector field. By definition
$\pr_D(\varepsilon)$ is hence an Euler-like derivation.  
\end{proof}

Let us fix now a $H$-twisted Dirac-Jacobi structure $\mathcal{L}\subseteq \mathbb{D}L$ for a line bundle $L\to M$. Let us 
also consider a transversal $\iota\colon N\hookrightarrow M $ and a section $\varepsilon=(\Delta,\alpha) \in 
\Secinfty(\mathcal{L})$, 
such that $\varepsilon \at{N}=0$ and $\Delta$ is an Euler-like derivation. Due to the Lemma \ref{Ex: UnFatTubNei}, we find a 
unique fat tubular neighbourhood 
	\begin{center}
		\begin{tikzcd}
 		L_\nu \arrow[r, "\Psi"]\arrow[d]&  L_U\arrow[d]   \\
 		\nu_N \arrow[r, "\psi"] & U
		\end{tikzcd}
	\end{center}	
such that $\Psi^*\Delta=\Delta_\mathcal{E}$. With this we have now two ways to construct a Dirac-Jacobi bundle on $L_\nu\to 
\nu_N$, namely we can take the Backwards transformation $\mathfrak{B}_\Psi(\mathcal L_U)$ and, if  we consider 
	\begin{center}
		\begin{tikzcd}
	 	L_\nu \arrow[r, "P"]\arrow[d]&   L_N \arrow[r, "I"]\arrow[d] & L\arrow[d] \\
 		\nu_N \arrow[r] &  N \arrow[r]& M
		\end{tikzcd}
		,
	\end{center}	
taking the backwards transform $\mathfrak{B}_{I\circ P}(\mathcal{L})=\mathfrak{B}_{P}(\mathfrak{B}_I(\mathcal{L}))$.
The aim is now to compare this two structures. Let us therefore consider the flow of $[\![(\Delta,\alpha) ,-]\!]_H$, which is given by 
	\begin{align*}
	( \gamma_t,\Phi^\Delta_t)\in Z^2_L(M)\rtimes \Aut(L),
	\end{align*}
where $\Phi^\Delta_t$ is the flow of $\Delta$ and  $\gamma_t=\int_0^t (\Phi_{-\tau}^\Delta)^*
(\D_L\alpha+\iota_\Delta H)\D\tau$.
For sure we have that the action of $( \gamma_t,\Phi^\Delta_t)$ preserves $\mathcal{L}$, which is explicitly 
	\begin{align*}
	\exp(\gamma_t)\circ \mathbb{D}\Phi^\Delta_t(\mathcal{L})=\mathcal{L}.
	\end{align*}
This leads us directly to the following theorem
\begin{theorem}[Normal form for Dirac-Jacobi bundles]\label{Thm: NormFormDJ}
Let $L\to M$ be a line bundle, let $H\in \Omega_L^3(M)$ be closed, let  $\mathcal{L}\subseteq \mathbb{D}L$ be a
$H$-twisted Dirac-Jacobi bundle and let $N\hookrightarrow M$ be a transversal. Then there exists an open neighbourhood $U\subseteq M$ of $N$ and  fat tubular neighbourhood $\Psi\colon L_\nu\to L_U$, such that 
	\begin{align*}
	\mathfrak{B}_\Psi(\mathcal{L}\at{U})=(\mathfrak{B}_{I\circ P}(\mathcal{L}))^\omega
	\end{align*}
for an $\omega\in \Omega^2_{L_\nu}(\nu_N)$.
\end{theorem}

\begin{proof}
According to Proposition \ref{Prop: ExEulerlikeDer}, we can find $(\Delta,\alpha)\in \Secinfty(\mathcal{L})$, such that 
$\Delta$ is Euler-like. Then there is a unique fat tubular neighbourhood $\Psi\colon L_\nu\to L_U$, such that 
$\Psi^*\Delta=\Delta_\mathcal{E}$, due to Lemma \ref{Ex: UnFatTubNei}. Let us denote by $ ( \gamma_t,\Phi^\Delta_t)\in Z^2_L(M)\rtimes \Aut(L)$ the flow of $[\![(\Delta,\alpha) ,-]\!]_H$. We know that $( \gamma_t,\Phi^\Delta_t)$ preserves 
$\mathcal{L}$ for all $t\in \mathbb{R}$ and so will $(\gamma_{-\log(s)},\Phi^\Delta_{-\log(s)})$ for all $s>0$.
Let us take a closer look to 
	\begin{align*}
	\gamma_{-\log(s)}&=\int_0^{-\log(s)}(\Phi_{-\tau}^\Delta)^*(\D_L\alpha+\iota_\Delta H)\D\tau\\&
	=\int_{-\log(1)}^{-\log(s)}(\Phi_{-\tau}^\Delta)^*(\D_L\alpha+\iota_\Delta H)\D\tau\\&
	=\int_{s}^{1}\frac{1}{t}(\Phi_{\log(t)}^\Delta)^*(\D_L\alpha+\iota_\Delta H)\D t
	\end{align*}	 
and we obtain that it is smoothly extendable to $s=0$ and let us denote its limit $s\to 0$ by $\omega'$ and 
$\omega=\Psi^*\omega'$. We have 
	\begin{align*}
	\mathfrak{B}_\Psi(\mathcal{L}\at{U})&
	=\mathfrak{B}_\Psi(\exp(\gamma_{-\log(s)})\circ \mathbb{D}\Phi^\Delta_{-\log(s)}(\mathcal{L}))\\&
	=\mathfrak{B}_\Psi(\exp(\gamma_{-\log(s)}) \mathfrak{B}_{\Phi^\Delta_{\log(s)}}(\mathcal{L}))\\&
	=(\mathfrak{B}_\Psi(\mathfrak{B}_{\Lambda_s}(\mathcal{L}))^{\Psi^*\gamma_{-\log(s)}}\\&
	=(\mathfrak{B}_{\Lambda_s\circ\Psi}(\mathcal{L}))^{\Psi^*\gamma_{-\log(s)}}\\&
	=(\mathfrak{B}_{\Psi\circ P_s}(\mathcal{L}))^{\Psi^*\gamma_{-\log(s)}}.
	\end{align*}
which holds for all $s\geq 0$. Hence we have for $s=0$, using that for the canonical inclusion $J\colon L_N\to L_\nu$ we have that $P_0=J\circ P$ and $\Psi\circ J=I$, that 
	\begin{align*}
	\mathfrak{B}_\Psi(\mathcal{L}\at{U})=(\mathfrak{B}_{I\circ P}(\mathcal{L}))^{\omega}.
	\end{align*}
\end{proof}

Note that this Theorem says, that up to a $B$-field, the Dirac-Jacobi structure is fully encoded in a given transversal, and 
hence the term "normal form" is justified by this fact. Moreover, it is possible to distinguish two different 
kind of leaves in Dirac-Jacobi geometry, see \cite{DirJacBun}, so it is also possible to distinguih two kinds of 
transversals,
which are more interesting in the Jacobi setting, since in the general Dirac-Jacobi setting the normal forms will be the 
same. Nevertheless, we will introduce them here and use them more excessively in the next section.

\begin{definition}[Cosymplectic Transversal]
Let $L\to M$ be a line bundle and let $\mathcal{L}\in \mathbb{D}L$ be a Dirac-Jacobi structure.
 A transversal $\iota\colon N\hookrightarrow M$ is called cosymplectic, if 
	\begin{align*}
	DL_N\cap \mathfrak{B}_I(\mathcal{L})=\{0\}.
	\end{align*}
\end{definition} 

 \begin{remark}
Note that a cosymplectic transversal always inherts a Dirac-Jacobi bundle coming from a Jacobi tensor by Proposition 
\ref{Prop: DJtoJ}. So let us denote $\mathcal{L}_{J_N}=\mathfrak{B}_I(\mathcal{L}_J)\subseteq \mathbb{D}L_N$.
\end{remark}

This transversals naturally appear as minimal transversal to locally conformal pre-symplectic leaves, see \cite{DirJacBun} for a more detailed discussion. 

So a corollary of this normal form theorem using the new notion of cosymplectic transversals

\begin{corollary}
Let $L\to M$ be a line bundle, let  $\mathcal{L}\subseteq \mathbb{D}L$ be a Dirac-Jacobi structure and let 
$\iota\colon N\hookrightarrow M$ be a minimal transversal 
to $\mathcal{L}$ at a locally conformal pre-symplectic point $p_0$, i.e. 
$\sigma(\pr_D(\mathcal{L}))\at{p_0}\oplus T_{p_0}N=T_{p_0}M$ and let $\nu_N=V\times N$. Then locally around $p_0$: 
	\begin{align*}
	\mathfrak{B}_\Psi(\mathcal{L}\at{U})=\{ (v+J_N^\sharp(\psi), \alpha +\psi)\in DL_\nu \ | \ 
	v\in TV , \alpha \in (\mathrm{Ann}(T^*V))\tensor L_\nu \text{ and } \psi\in J^1L_N\} ^\omega
	\end{align*}
where $J_N$ is the Jacobi structure on the transversal and the canonical identification $DL_{\nu_N}= TV\oplus DL_N$.
\end{corollary}

\begin{proof}
Note that it is easy to check that for a minimal transversal $N$ at a locally conformal pre-symplectic point $p_0$ the equation 	
	\begin{align*}
	DL_N\cap \mathfrak{B}_I(\mathcal{L})=\{0\}
	\end{align*}
holds at $p_0$ and hence in a whole neighbourhood. The rest is an application of Theorem \ref{Thm: NormFormDJ} and the usage of the splitting $DL_{\nu_N}= TV\oplus DL_N$.
\end{proof}

The other kind of leaves of a Dirac-Jacobi structure are so-called pre-contact leaves. Their minimal transversal posses the
following structure :

\begin{definition}[Cocontact Transversal]\label{Cocontact}
Let $L\to M$ be a line bundle and let $\mathcal{L}\in \mathbb{D}L$ be a Dirac-Jacobi structure.
 A transversal $\iota\colon N\hookrightarrow M$ is called cocontact, if 
	\begin{align*}
	\rank(DL_N\cap \mathfrak{B}_I(\mathcal{L}))=1.
	\end{align*}
\end{definition}

\begin{lemma}\label{Lem: MinTrnsConLf}
Let $L\to M$ be a line bundle, let  $\mathcal{L}\subseteq \mathbb{D}L$ be a Dirac-Jacobi structure and let 
$\iota\colon N\hookrightarrow M$ be a minimal transversal 
to $\mathcal{L}$ at a pre-contact point $p_0$. Then 
	\begin{align*}
	\rank(DL_N\cap \mathfrak{B}_I(\mathcal{L}))=1
	\end{align*}
holds in a neighbourhood of $p_0$.
\end{lemma}

\begin{proof}
It is easy to see that 
	\begin{align*}
	(DL_N\cap \mathfrak{B}_I(\mathcal{L}))\at{p_0}=\langle \mathbbm{1}\rangle
	\end{align*}
Now we want to argue why this holds in a whole neighbourhood. Let us therefore consider the sum 
$DL_N+ \mathfrak{B}_I(\mathcal{L})\subseteq \mathbb{D}L$ and a (local) section $\alpha\in \Omega^1_L(M)$ such that 
$\alpha(\mathbbm{1})\at{p_0} \neq 0$. Let $(0,\beta)\in (DL_N+ \mathfrak{B}_I(\mathcal{L}))\at{p_0}\cap 
\langle \alpha \rangle\at{p_0}$, then there exists $\Delta\in D_{p_0}L$ such that 
$(\Delta,\beta)\in \mathfrak{B}_I(\mathcal{L})$, but since $(\mathbbm{1},0)\in \mathfrak{B}_I(\mathcal{L})$, 
we have using the isotropy of $\mathfrak{B}_I(\mathcal{L})$,  
	\begin{align*}
	0=\bla (\Delta,\beta), (\mathbbm{1},0)\bra =\beta(\mathbbm{1}).
	\end{align*}	 
and hence, for dimensional reasons,  $\mathbb{D}L\at{p_0}=(DL_N+ \mathfrak{B}_I(\mathcal{L}))\at{p_0}\oplus 
\langle \alpha \rangle\at{p_0}$. Therefore this equality holds in a whole neighbourhood of $p_0$, 
so $\rank(DL_N+ \mathfrak{B}_I(\mathcal{L}))=2n+1$ in this neighbourhood, which implies 
$\rank(DL_N\cap \mathfrak{B}_I(\mathcal{L}))=1$ around $p_0$.
\end{proof}

\begin{remark}
Note that  a cocontact transversal does not inhert a Jacobi structure, but nevertheless the Dirac-Jacobi structure is of 
homogeneous Poisson type. 
\end{remark}

\begin{definition}
Let $L\to M$ be a line bundle and let  $\mathcal{L}\in \mathbb{D}L$ be a Dirac-Jacobi structure. A homogeneous cocontact transversal $\iota\colon N\hookrightarrow M$ is a cocontact transversal together with a flat connection $\nabla\colon TN\to DL_N$, such that 
	\begin{align*}
	\image(\nabla)\oplus (DL_N\cap \mathfrak{B}_I(\mathcal{L}))=DL_N.
	\end{align*}
\end{definition}

\begin{remark}
The definition of a homogeneous cocontact transversal seems a bit strange, since it includes a connection. This fact can be 
explained quite easily using the homogenezation described in \cite{DirJacBun}, which turns a Dirac-Jacobi structure on a line 
bundle $L\to M$ into a Dirac structure on $L^\times:=L^*\backslash \{0_M\}$ which is homogeneous (in the sense of 
\cite{2017arXiv171108310S}) with respect to the shrinked 
Euler vector field  $\mathcal{E}$ on $L^*$. The pre-symplectic leaves of this Dirac structure have the additional 
property that $\mathcal{E}$ is either tangential to it or transversal. If $\mathcal{E}$ is tangential, then the leaf 
corresponds to a pre-contact leaf on the base $M$. Hence a minimal transversal $N$ to it is transversal to the 
Euler vector field and  defines therefore a horizontal bundle on $L_{\pr(N)}^*$ and hence a connection.       
\end{remark}

\begin{proposition}
Let $L\to M$ be a line bundle, let  $\mathcal{L}\subseteq \mathbb{D}L$ be a Dirac-Jacobi structure and let 
$\iota\colon N\hookrightarrow M$ be a minimal transversal 
to $\mathcal{L}$ at a pre-contact point $p_0$. Then every flat connection $\nabla$ gives $N$ locally the structure of a 
homogeneous cocontact transversal.
\end{proposition}

\begin{proof}
In the proof of Lemma \ref{Lem: MinTrnsConLf}, we have seen that 
	\begin{align*}
	(DL_N\cap \mathfrak{B}_I(\mathcal{L}))\at{p_0}=\langle \mathbbm{1}\rangle
	\end{align*}
and hence for every flat connection $\nabla$, we have that $\image(\nabla)\at{p_0}\oplus 
(DL_N\cap \mathfrak{B}_I(\mathcal{L}))\at{p_0}=DL_N$ and hence this decomposition holds in a whole neighbourhood of $p_0$. 
\end{proof}

An immediate consequence is:

\begin{corollary}
Let $L\to M$ be a line bundle, let  $\mathcal{L}\subseteq \mathbb{D}L$ be a Dirac-Jacobi structure and let 
$\iota\colon N\hookrightarrow M$ be a homogeneous cocontact transversal with connection $\nabla$. Then there exists a local trivialization of $\L_\nu$ such that, using the to $\nabla $ corresponding trivializations $DL_\nu=T\nu\oplus\mathbb{R}_M$ and 
$J^1L=T^*M\oplus \mathbb{R}_M$,  
	\begin{align*}
	\mathfrak{B}_\Psi(\mathcal{L}\at{U})=\{(v+r(\mathbbm{1}-{Z_N})+ \pi^\sharp(\psi), \alpha+ \psi+ \psi(Z_N)\mathbbm{1}^*)\ | \ v\in TV , \alpha \in \mathrm{Ann}(T^*V) \text{ and } \psi\in T^*N\} ^{\omega},
	\end{align*}
where $(\pi_N,Z_N)$ is the homogeneous Poisson structure on the transversal from Lemma \ref{Lem: HomPoi}.
\end{corollary}

This last two corollaries can be seen as the Jacobi-geometric analogue of the results obtained by Blohmann in 
\cite{B2017}.

\section{Normal forms and Splitting Theorems of Jacobi bundles}\label{Sec: NormForJ}
As explained in Example \ref{Ex: Jacobi}, Jacobi bundles are a special kind of Dirac-Jacobi structures. In addition, we have 
that Jacobi isomorphism induces an isomorphism of the corrsponding Dirac structures (this holds even for morphisms if one 
considers forward maps of Dirac-Jacobi structures which we will not explain here, see \cite{DirJacBun}). The converse is 
unfortunately not true:
if the Dirac-Jacobi structures of two Jacobi structures are isomorphic, it does not follow in general that the Jacobi 
structures are isomorphic. The parts which are not "allowed" in Jacobi geometry are the $B$-fields. Nevertheless, we can keep 
track of them, if we make further assumptions on the transversals, namely cosymplectic and cocontact transversals.

\subsection{Cosymplectic Transversals}
In this part, we are using the notion of cosymplectic transversals as explained in the previous section. The difference is now that in Jacobi geoemtry this transversal gives us more than on arbitrary Dirac-Jacobi manifolds. In fact, the Jacobi structure induces a line bundle valued symplectic structure on the normal bundle, to be seen in the following

\begin{lemma}\label{Lem: CanSplit}
Let $L\to M$ be a line bundle, $J\in \Secinfty(\Anti^2(J^1L)^*\tensor L)$ be a Jacobi tensor with corresponding Dirac-
Jacobi structure $\mathcal{L}_J\in \mathbb{D}L$ and let   $\iota\colon N\hookrightarrow M$ be a  cosymplectic transversal. Then 
	\begin{align*}
	J^\sharp(\mathrm{Ann}(DL_N))\oplus DL_N=DL\at{N}.
	\end{align*}
\end{lemma}

\begin{proof}
First we prove that $J^\sharp\at{\mathrm{Ann}(DL_N)}$ is injective. Let therefore $\alpha\in \mathrm{Ann}(DL_N)$, such that 
$J^\sharp(\alpha)=0$. Hence we have for an arbitrary $\beta\in J^1L$, that 
	\begin{align*}
	\alpha(J^\sharp(\beta))=-\beta(J^\sharp(\alpha))=0.
	\end{align*}
Hence $\alpha=\mathrm{Ann}(DL_N)\cap \mathrm{Ann}(\image(J^\sharp))=\mathrm{Ann}(DL_N+\image(J^\sharp))=\{0\}$, and  
$J^\sharp\at{\mathrm{Ann}(DL_N)}$ is injective. Let $\Delta\in DL_N\cap J^\sharp(\mathrm{Ann}(DL_N))$, then there exists an
$\alpha\in \mathrm{Ann}(DL_N)$, such that $J^\sharp(\alpha)=\Delta$. Thus, we have that $(\Delta,\alpha)\in \mathcal{L}_J$ 
and moreover $(\Delta,DI^*\alpha)\in \mathfrak{B}_I(\mathcal{L}_J)$, but since $\alpha\in \mathrm{Ann}(DL_N)$, we have that 
$DI^*\alpha=0$ and hence $\Delta=0$, since $N$ is cosymplectic. Counting dimensions the claim follows.
\end{proof}

Suppose that $\iota\colon N\hookrightarrow M$ is a cosymplectic transversal, then we have that 
	\begin{align*}
	\pr_\nu\circ \sigma\circ J^\sharp\colon \mathrm{Ann}(DL_N)\to \nu_N
	\end{align*}
is an isomorphism. Let us chose $\alpha\in \Secinfty( J^1L)$, such that $\alpha\at{N}=0$ and such that 
$\D^N\alpha\colon \nu_N\to \mathrm{Ann}(DL_N)\subseteq J^1L\at{N}$ is a right-inverse to 
$\pr_\nu\circ \sigma\circ J^\sharp$. We have then
	\begin{align*}
	\pr_\nu(\D^N \sigma(J^\sharp(\alpha)))= \pr_\nu(\sigma(J^\sharp(\D^N\alpha)))=\id_{\nu_N}
	\end{align*}
and hence we have that $T\nu(\sigma(J^\sharp(\alpha)))=\mathcal{E}$. Multiplying $\alpha$ by a bump-function, which is 1 near
$N$, we may arrange that $\sigma(J^\sharp(\alpha))$ is complete and hence $J^\sharp(\alpha)$ is an Euler-like derivation. By 
Theorem \ref{Thm: NormFormDJ}, we have that 
	\begin{align*}
	\mathfrak{B}_\Psi(\mathcal{L}_J)=\mathfrak{B}_P(\mathcal{L}_{J_N})^\omega,
	\end{align*}			
where $\omega = \Psi^*\int_{0}^{1}\frac{1}{t}(\Phi_{\log(t)}^\Delta)^*(\D_L\alpha)\D t$ and $\Psi\colon L_\nu\to L_U$ is the unique tubular neighbourhood, such that $\Psi^*(J^\sharp(\alpha))=\Delta_\mathcal{E}$. 

\begin{proposition}\label{Prop: Propertiesomega}
The 2-form $\omega\in \Omega_{L_\nu}^2(\nu_N)$ shrinked to $N$ has kernel $DL_N$.
\end{proposition}
\begin{proof}
One can show, in local coordinates,  that $\D^N\alpha([\sigma(\Box)\at{N}])=\Lie_\Box\alpha\at{N}$ for all
$\Box\in \Secinfty(DL)$. Hence we have trivially $\Lie_\Delta\alpha\at{N}=0$ for $\Delta\in \Secinfty(DL)$, such that $\Delta\at{N}\in\Secinfty(DL_N)$. Let now $\Delta,\Box\in \Secinfty(DL)$, such that 
$\Delta\at{N}\in\Secinfty(DL_N)$, then
	\begin{align*}
	\D_L\alpha(\Delta,\Box)\at{N}&
	=-(\D_L\iota_\Delta\alpha)(\Box)\at{N}=-\Box(\alpha(\Delta))\at{N}\\&
	=-(\Lie_\Box\alpha)(\Delta)\at{N}-\alpha([\Box,\Delta])\at{N}
	= -(\Lie_\Box\alpha)\at{N}(\Delta)\\&
	=\D^N\alpha([\sigma(\Box)\at{N}])(\Delta)\\&
	=0,
	\end{align*}
where the last equality follows since $\D^N\alpha\colon \nu_N\to \mathrm{Ann}(DL_N)$. Hence we have that $\ker((\D_L\alpha)^\flat)\supseteq DL_N$, in particular this is true for $\frac{1}{t}(\Phi_{\log(t)}^\Delta)^*(\D_L\alpha)$, since 
$\Phi_{\log(s)}\at{N}$ is a gauge transformation fixing $DL_N$. Thus it is true also for $\omega$, since 
$D\Psi\at{DL_N}=\id$.
\end{proof}
We want to describe the structure of $\omega$ at $N$. Note that for a cosymplectic transversal $N$, the normal bundle 
always comes together with a canonical symplectic (i.e. non-degenerate) $L_N$-valued 2-form $\Theta\in \Secinfty(\Anti^2 \nu_N^*\tensor L_N)$ defined by 
	\begin{align*}
	\Theta(X,Y)= (\pr_\nu\circ \sigma\circ J^\sharp\at{\mathrm{Ann}(DL_N)})^{-1}(X)(Y)
	\end{align*}
\begin{lemma}\label{Lem: PropOme II}
The 2-form $\omega\in \Omega_{L_{\nu}}^2(\nu_N)$ coincides, shrinked to $\nu_N\subseteq DL_{\nu_N}$, with $\Theta$. 
\end{lemma}

\begin{proof}
Note that for a cosymplectic transversal, we have
	\begin{align*}
	DL\at{N}=DL_N\oplus J^\sharp(\mathrm{Ann}(DL_N))=DL_N\oplus \nu_N
	\end{align*}	 
with the canonical identification 
	\begin{align*}
	J^\sharp(\mathrm{Ann}(DL_N))=\frac{DL\at{N}}{DL_N}=\nu_N.
	\end{align*}
Moreover, we have 
	\begin{align*}
	DL_\nu\at{N}=DL_N\oplus \nu_N, 
	\end{align*}
where we include $\nu_N$ by the following map:
	\begin{align*}
	\chi\colon \nu_N\ni v_n \to \big(\lambda \to \frac{\D}{\D t}\At{t=0} P_0\lambda(p_t(v_p))\big) \in D_n L_\nu. 
	\end{align*}
It is clear that $D\Psi$ fixes $DL_N$, since $\Psi\at{N}\colon L_N\to L_N$ is identity. We want to show that 
$D\Psi(\nu_N)\subseteq J^\sharp(\mathrm{Ann}(DL_N))$. One can show that by an elementary calculation, that 
	\begin{align*}
	D\Psi (\chi(v_n))= \lim_{t\to 0} \frac{\Delta_{\lambda_t(\psi(v_n))}}{t}
	\end{align*}
using Equation \ref{Eq: CommutationD}. But by defintion, we have that 
	\begin{align*}
	\D^N\Delta(v_n)= \lim_{t\to 0} \frac{\Delta_{\lambda_t(\psi(v_n))}}{t}
	\end{align*}
hence  $D\Psi\circ\chi=\D^N\Delta = J^\sharp\circ \D^N\alpha$, but $\alpha$ was chosen in such a way that 
$\D^N\alpha$ takes values in $\mathrm{Ann}(DL_N)$. Thus $D\Psi\at{N}$ respects the splitting. 
Using this and 
 	\begin{align*}
 	\mathfrak{B}_\Psi(\mathcal{L}_J)=\mathfrak{B}_P(\mathcal{L}_{J_N})^\omega
 	\end{align*}
and $\ker(\omega^\flat)\at{N}=DL_N$ and the definition of $\Theta$, we see that at $N$ they have to coincide.
\end{proof}
This leads us to the normal form theorem for Jacobi manifolds.

\begin{theorem}[Normal Form  for Jacobi bundles I]\label{Thm: LocNorFor I}
Let $L\to M$ be a line bundle, let $J$ be a Jacobi structure and let $N\to M$ be a cosymplectic transversal. For a closed 
2-form 
$\omega\in \Omega_{L_\nu}^2(\nu_N)$, such that $\ker(\omega^\flat)\at{N}=DL_N$ and $\omega$ coincides with $\Theta$ at $\nu_N\subseteq DL_\nu$. Then
\begin{align*}
\mathfrak{B}_P(\mathcal{L}_{J_N})^\omega
\end{align*}
is the graph of a Jacobi structure near the zero section and there exists a fat tubular neighbourhood $\Psi\colon L_\nu\to L_U$ which is a Jacobi map near the zero section. 
\end{theorem}

\begin{proof}
We have proven this theorem for the special $\omega$ given by 
	\begin{align*}
	\omega=\int_{0}^{1}\frac{1}{t}(\Phi_{\log(t)}^\Delta)^*\D_L\alpha\D t.
	\end{align*}
Let $\omega'$ be a second 2-form fulfilling the requirements of the theorem, then
\begin{align*}
\sigma_t:=t(\omega'-\omega)
\end{align*}
is a (time-dependent) 2-form such that $\sigma_0=0$ and moreover $\sigma_t\at{N}=0$. Thus,
	\begin{align*}
	(\mathfrak{B}_P(\mathcal{L}_{J_N})^\omega)^{\sigma_t}=\mathfrak{B}_P(\mathcal{L}_{J_N})^{\omega+\sigma_t}
	\end{align*}	 
is a Jacobi structure near $N$. Now we can apply Appendix \ref{App: Moser} to get the result. 
\end{proof}

An immediaty consequence of this theorem is the Splitting for Jacobi manifolds around a locally conformal symplectic leaf, proven by Dazord, Lichnerowicz and Marle in \cite{SplittingThmJac}. 

\begin{theorem}\label{Thm: SplittingI}
Let $L\to M $ be a line bundle, let $J\in \Secinfty(\Anti^2 (J^1L)^*\tensor L)$ be a Jacobi tensor and let $p_0\in M$ be a 
locally conformal symplectic point. 
Then there are a line bundle trivialization $L_U\cong U\times \mathbb{R}$ around $p_0$ and a cosymplectic transversal $N
\hookrightarrow U$, such that $U\cong U_{2q}\times N$ for an open subset $0\in U_{2q}\subseteq \mathbb{R}^{2q}$ and the corresponding Jacobi pair $(\Lambda,E)$ is transformed (via
this isomorphism) to 
	\begin{align*}
	(\Lambda,E)=(\pi_{\mathrm{can}}+\Lambda_N+ E_N\wedge Z_{\mathrm{can}}, E_N),
	\end{align*}
where $(\Lambda_N, E_N)$ is the induced Jacobi structure on the transversal $N$ and the canonical stuctures on the fiber are given by 
$(\pi_{\mathrm{can}}, Z_{\mathrm{can}})=
(\frac{\partial}{\partial p_i}\wedge \frac{\partial}{\partial q^i},p_i\frac{\partial}{\partial p_i})$. 	  
\end{theorem}

\begin{proof}
We can assume from the beginning that the line bundle is trivial, since otherwise we can trivialize around $p_0$ and and shrink 
the line bundle to this open neighbourhood. Let us choose an arbitrary transversal $N$ to the leaf $S$ at $p_0$ (in the sense, that $S\times N = M$). It is easy to see that  
	\begin{align*}
	(DL_N\cap \mathfrak{B}_I(\mathcal{L}_J))\at{p}=\{0\},
	\end{align*}
and hence we can shrink to an open neighbourhood of $p_0$, where this equality holds. This means every transversal to a leaf is  a cosymplectic transversal near the intersection point. Let us from now on denote $p_0=(s_0,n_0)$, hence 
$\nu_N\cong T_{s_0}S \times N \cong \mathbb{R}^{2k}\times N$. Since the line bundle is trivial, we can identify $\nu_N$ together with $\Theta$ as a 
symplectic vector bundle, hence we find a possible smaller $N$ and a vector bundle automorphism of $\nu_N$, such that $\Theta$ is the constant symplectic form.  We can now choose 
	\begin{align*}
	\omega=\D q^i\wedge \D p_i - \mathbbm{1}^*\wedge p_i\D q^i \in \Omega_{L_\nu}(\nu_N)
	\end{align*}
where $(q,p)$ are the symplectic coordinates on $\nu_N\to N$. This 2-from is $\D_L$-closed and coincides with $\Theta$ on 
$N$,
moreover  $\ker(\omega^\flat)\at{N}=DL_N$. Hence the requirements of  Theorem \ref{Thm: LocNorFor I} are fulfilled and the claim follows by an easy computation. 
\end{proof}

\subsection{Cocontact transversals}
The second kind of transversals we want to discuss in the context of Jacobi geometry are cocontact transversals, which were also introduced before in 
Definition \ref{Cocontact}. In fact this notion is not enough for our purposes  and  we need to assume more information on the structure of the transversal, which is precisely the notion of homogeneous cocontact transversal from Definition
 \ref{Cocontact}.

\begin{lemma}
Let $L\to M$ be a line bundle, $J\in \Secinfty(\Anti^2(J^1L)^*\tensor L)$ be a Jacobi tensor with corresponding Dirac-
Jacobi structure $\mathcal{L}_J\in \mathbb{D}L$ and let   $\iota\colon N\hookrightarrow M$ be a 
homogeneous cocontact transversal with connection $\nabla\colon TN\to DL_N$. Then 
	\begin{align*}
	J^\sharp(\mathrm{Ann}(\image(\nabla)))\oplus \image(\nabla)=DL\at{N}.
	\end{align*}
\end{lemma}

\begin{proof}
The proof follows the same lines as Lemma \ref{Lem: CanSplit}.
\end{proof}
We pick now, as in the cosymplectic case, an $\alpha\in \Secinfty(J^1L)$, such that $\alpha\at{N}=0$ and 
	\begin{align*}
	\D^N\alpha\colon \nu_N\to \mathrm{Ann}(\image(\nabla))\subseteq J^1L\at{N}
	\end{align*}
defines a splitting of $I^!\mathcal{L}\to \mathcal{L}\at{N}\to \nu_N$, i.e. 
$\pr_\nu\circ \sigma\circ J^\sharp\circ \D^N\alpha=\id_{\nu_N}$. Hence we have that $J^\sharp(\alpha)$, multiplied by a 
suitable bump function which is $1$ close to $N$, is an Euler-like derivation. By 
Theorem \ref{Thm: NormFormDJ}, we have that 
	\begin{align*}
	\mathfrak{B}_\Psi(\mathcal{L}_J)=\mathfrak{B}_P(\mathfrak{B}_I(\mathcal{L}))^\omega,
	\end{align*}			
where $\omega = \Psi^*\int_{0}^{1}\frac{1}{t}(\Phi_{\log(t)}^\Delta)^*(\D_L\alpha)\D t$ and $\Psi\colon L_\nu\to L_U$ is the 
unique tubular neighbourhood, such that $\Psi^*(J^\sharp(\alpha))=\Delta_\mathcal{E}$. We can prove, as before, 
the following
\begin{proposition}
The 2-form $\omega\in \Omega_{L_\nu}^2(\nu_N)$ shrinked to $N$ has kernel $\image(\nabla)$.
\end{proposition}
\begin{proof}
This proof follows the same lines as the proof of Proposition \ref{Prop: Propertiesomega}.
\end{proof}

As in the cosymplectic transversal case, we can define a skew symmetric 2-form
	\begin{align*}
	\Theta \in \Secinfty(\Anti^2 J^\sharp (\mathrm{Ann}(\image(\nabla))\tensor L_N)
	\end{align*}
by 
	\begin{align*}
	\Theta(X,Y)= (J^\sharp\at{\mathrm{Ann}(\image(\nabla))})^{-1}(X)(Y).
	\end{align*}
It is easy to see that $\Theta$ is non-degenerate. Moreover, we have 

\begin{lemma}
The 2-form $\omega\in \Omega_{L_{\nu}}^2(\nu_N)$ coincides, shrinked to $\nu_N\oplus K \subseteq DL_{\nu_N}$, with $\Theta$,
where we denote $K:=(DL_N\cap \mathfrak{B}_I(\mathcal{L}_J))$. 
\end{lemma}

\begin{proof}
Using the ideas of the proof of Lemma \ref{Lem: PropOme II}, we can show that the fat tubular neighbourhood transports
$J^\sharp(\mathrm{Ann}(\image(\nabla))$ to $\nu_N\oplus K$, hence the proof is copy and paste of this Lemma.
\end{proof}

\begin{theorem}[Normal Form  for Jacobi bundles II]\label{Thm: LocNorFor II}
Let $L\to M$ be a line bundle, let $J$ be a Jacobi structure and let $N\to M$ be a cocontact transversal with connection 
$\nabla\colon TN\to DL_N$. For a closed 
2-form 
$\omega\in \Omega_{L_\nu}^2(\nu_N)$, such that $\ker(\omega^\flat)\at{N}=\image(\nabla)$ and $\omega$ coincides with $\Theta$ 
at $\nu_N\oplus (\mathfrak{B}_I (\mathcal{L}_J)\cap DL_N)\subseteq DL_\nu$. Then
\begin{align*}
\mathfrak{B}_P(\mathcal{L}_N)^\omega
\end{align*}
is the graph of a Jacobi structure near the zero section and there exists a fat tubular neighbourhood $\Psi\colon L_\nu\to 
L_U$ which is a Jacobi map near the zero section. 
\end{theorem}

\begin{proof}
The proof follows the lines of Theorem \ref{Thm: LocNorFor I} with the obvious adaptions. 
\end{proof}

The next step is to prove the second splitting Theorem of Dazord and Lichnerowicz and Marle in \cite{SplittingThmJac}, namely 
the splitting of Jacobi manifolds around contact leaves. 

\begin{theorem}\label{Thm: SplittingII}
Let $L\to M $ be a line bundle, let $J\in \Secinfty(\Anti^2 (J^1L)^*\tensor L)$ be a Jacobi tensor and let $p_0\in M$ be a 
contact point. 
Then there are a line bundle trivialization $L_U\cong U\times \mathbb{R}$ around $p_0$ and a homogeneous cocontact 
transversal 
$N\hookrightarrow U$, such that $U\cong U_{2q+1}\times N$ for an open subset $0\in U_{2q+1}\subseteq \mathbb{R}^{2q+1}$ and the 
corresponding Jacobi pair $(\Lambda,E)$ is transformed (via
this isomorphism) to 
	\begin{align*}
	(\Lambda,E)=(\Lambda_{\mathrm{can}}+\pi_N+ E_\mathrm{can}\wedge Z_N, E_\mathrm{can}),
	\end{align*}
where $(\pi_N, Z_N)$ is the induced homogeneous Poisson structure on the transversal $N$ and the contact structure on the fiber is given by $(\Lambda_{\mathrm{can}}, E_\mathrm{can})
=((p_i\frac{\partial}{\partial u}+\frac{\partial}{\partial q^i})
\wedge \frac{\partial}{\partial q_i},\frac{\partial}{\partial u})$ . 	  
\end{theorem}

\begin{proof}
Let $p_0\in M$ be a contact point and let $N\subseteq M$ be a transversal, such that 
	\begin{align*}
	\sigma(\image J^\sharp)\at{p_0} \oplus T_{p_0}N=T_{p_0}M.
	\end{align*}	 
We can again assume that the line bundle $L\to M$ is trivial, since we want to prove a local statement. 
In a possibly smaller neighbourhood, we can assume that also the normal bundle $\nu_N=V\times N\to N$ is trivial. 
We want to show that there is a trivialization of $\nu_N$, such that $\Theta$ looks trivial, where we specialize on the way 
through the proof what we mean by trivial. Let us therefore denote by $\lambda$ the local trivializing section of $L_N$, thus 
we can write
	\begin{align*}
	\Theta(\Delta, \Box)=\Omega(\Delta, \Box)\cdot \lambda
	\end{align*}	  
for $\Delta,\Box \in \nu_N\oplus K$. Since $L_N\to N$ is trivial, we identify $DL_N= TN\oplus \mathbb{R}_N$ and choose the 
trivial connection $\nabla$. Hence, we can find a (local) nowhere vanishing section of $K$ of the form $\mathbbm{1}- Z$ for a 
unique $Z$. Let us now shrink 
	\begin{align*}
	\Theta\at{\nu_N}\colon \nu_N\times\nu_N\to L_N, 
	\end{align*}	 
since $\nu_N$ is odd dimensional and $\Theta$ is a skew-symmetric pairing, we can find a local non-vanishing
 $X\in \Secinfty(\nu_N)$, such 
that $\Theta(X,\cdot)=0$, moreover, since $\Theta$ is non-degenerate, we can modify $X$ in such a way that 
	\begin{align*}
	\Omega(\mathbbm{1}-Z, X)=1.
	\end{align*}	
It is  now easy to see that symplectic complement  $S:=\langle \mathbbm{1}-Z, X\rangle^\perp\subseteq \nu_N$. Finally, 
we find a trivialization of $S$ such that $\Omega\at{S}$ is the trivial symplectic form with Darboux frame 
$\{e_2,e_{k+2},\dots \}$. Hence, by extending this 
trivialization to $\nu_N=V\times N$ by using the coordinate $X$ as $e_0$, we find that 
$\{e_0,\mathbbm{1}-Z, e_1,e_{k+1},e_2,e_{k+2},
\dots\}$ is a Darboux frame of $\Omega$ in this trivialization. with the decomposition $DL_\nu=TV\oplus TN \oplus 
\mathbb{R}_{\nu_N}$ we can 
choose 
	\begin{align*}
	\omega=\sum_{i=1}^k dx^i\wedge dx^{i+k} + \mathbbm{1}^*\wedge (dx^0-\sum_{i=1}^k x^{i+k}dx^i)
	\end{align*}	   
which coincides with $\Theta$ on $\nu_N\oplus K$ and is $\D_L$-closed. By applying Theorem \ref{Thm: LocNorFor II}, since $N$ together with 
$\nabla$ is a homogeneous cocontact transversal, we find a Jacobi morphism  
	\begin{align*}
	\mathfrak{B}_{P}(\mathcal{L}_N)^\omega\cong \mathcal{L}_J.
	\end{align*}
An easy computation shows that $\mathfrak{B}_{P}(\mathcal{L}_N)^\omega$ is the graph of the Jacobi structure of the 
form in the theorem. 
\end{proof}

\section{Application: Splitting theorem for homogeneous Poisson Structures}
Using the homogenezation scheme from \cite{BGG2017}, one can see that Jacobi bundles are nothing else but special kinds of 
homogeneous Poisson manifolds. Moreover, the two most important examples of Poisson manifolds are of this kind: the cotangent 
bundle and the dual of a Lie algebra. 
Using this insight, it is easy to see that proving something for Jacobi structures gives a proof for something in homogeneous 
Poisson Geometry. We want to apply this philosophy to give a splitting theorem for homogeneous Poisson manifolds.
The first appearance of such a theorem was  \cite[Theorem 5.5]{SplittingThmJac} in order 
to prove the local splitting of 
Jacobi pairs. Here we want to attack the problem from the other side: we use the splitting of Jacobi manifolds to prove the 
splitting of homogeneous Poisson structures.  

\begin{theorem}
Let $(\pi,Z)$ be a homogeneous Poisson structure on a manifold $M$ and let $p_0\in M$ be a point such that $Z_{p_0}\neq 0$. Then there exist an open neighbourhood $U$ of $p_0$,  an open neighbourhood $U_{2k}$ of $0\in \mathbb{R}^{2k}$, 
a manifold $N$ with a homogeneous Poisson structure $(\pi_N,Z_N)$ and a diffeomorphism $\psi\colon U\to U_{2k}\times N$, such that 
	\begin{align*}
	\psi_* \pi = \frac{\partial}{\partial p_i}\wedge\frac{\partial}{\partial q^i} +\pi_N.
	\end{align*}
Additionally, 
	\begin{enumerate}
	\item if $Z\in \image(\pi^\sharp)$, then $\psi_*Z=p_i\frac{\partial}{\partial p_i} +\frac{\partial}{\partial p_k}+Z_N$.
	\item if $Z\notin \image(\pi^\sharp)$, then $\psi_*Z=p_i\frac{\partial}{\partial p_i}+Z_N$.
	\end{enumerate}

\end{theorem}

\begin{proof}
Note that since $Z_{p_0}\neq 0$, we find coordinates 
$\{u,x^1,\dots, x^q\}$ with $p_0=(1,0,\dots,0)$, such that $Z=u\frac{\partial}{\partial u}$. In this chart, we have, using $\Lie_Z\pi=-\pi$,  
	\begin{align*}
	\pi =\frac{1}{u}(\Lambda + u\frac{\partial}{\partial u}\wedge E) 
	\end{align*}	 
for unique $\Lambda\in \Secinfty(\Anti^2 TM)$ and $E\in \Secinfty(TM)$ which do not 
depend on $u$. It is easy to see, that we 
have 
	\begin{align*}
	[\Lambda,\Lambda]=-E\wedge \Lambda \text{ and } \Lie_E\Lambda=0,
	\end{align*}	 
which means that $(\Lambda,E)$ is a Jacobi pair. This allows us to use Theorem \ref{Thm: SplittingI} and 
Theorem \ref{Thm: SplittingII} to prove the result. We will do it just for the case where $p_0$ is a contact point, which 
means, translated to Jacobi pairs, that $E_{p_0}$ is transversal to $\image(\Lambda^\sharp)\at{p_0}$ and thus 
$Z\in \image(\pi^\sharp)$, since the other case is exactly the same.
 Note that, we can apply Theorem \ref{Thm: SplittingII}: there exists coordinates 
$\{x,q^i,p_i, y^j\}$ and a local non-vanishing function $a$( which is basically the line bundle trivialization), such that 
	\begin{align*}
	\Lambda = \frac{1}{a}(\Lambda_{\mathrm{can}}+\pi_N+ E_{\mathrm{can}}\wedge Z_N)
	\text{ and } 
	E=\frac{1}{a}(E_{\mathrm{can}}+\Lambda^\sharp(\D a)),
	\end{align*}
where 	$\Lambda_{\mathrm{can}}$ and $E_{\mathrm{can}} $ are just depending on $\{x,q^i,p_i\}$ and $(\phi_N,Z_N)$ is a homogeneous Poisson structure just depending on $y^j$-coordinates.

If we apply the diffeomorphism $(u,x^1,\dots,x^q)\mapsto (a\cdot u,x^1,\dots,x^q)$, we have 	
	\begin{align*}
	\pi=\frac{1}{u}(\Lambda_{\mathrm{can}}+\pi_N+ E_{\mathrm{can}}\wedge Z_N+u\frac{\partial}{\partial u}
	\wedge  E_{\mathrm{can}}).
	\end{align*}
A (quite) long and not very insightful computation shows that the  diffeomorphism 
\begin{align*}
\Phi(u,x^1,\dots,x^q)=(u,\Phi^{Z_N}_{\log(u)}(\Phi^{E_{\mathrm{can}}}_{-log(u)}(x^1,\dots,x^q))),
\end{align*}
where $\Phi^{Z_N}_t$ (resp. $\Phi^{E_{\mathrm{can}}}_t$) is the flow uf $Z_N$ (resp. $E_{\mathrm{can}}$),
gives us 
	\begin{align*}
	\pi = \frac{1}{u}( \frac{\partial}{\partial p_i}\wedge\frac{\partial}{\partial q^i})+ \frac{\partial}{\partial u}\wedge
	\frac{\partial}{\partial x}+\pi_N \text{ and } Z=u\frac{\partial}{\partial u}+p_i\frac{\partial}{\partial p_i}+Z_N
	\end{align*}	 
and with some obvious variations and renaming coordinates of $\pi$ we get the result. 
\end{proof}

This Application shows us that, eventhough we can see Poisson structures as Jacobi manifolds, which suggests that they are 
more general objects than Poisson structures, the splitting theorems (of Jacobi pairs) are a refinement of the known 
splitting theorems for Poisson structures.

\section{Generalized Contact bundles}
In this last section, we want to drop a word about generalized contact bundles. They were introduced recently in 
\cite{VW2016} and they are modeled to be the odd dimensional analogue to generalized complex structures. 

\begin{definition}\label{Def: GenCon}
Let $L\to M$ be a line bundle. A subbundle $\mathcal{L}\subseteq \mathbb{D}_\mathbb{C}L$ is called generalized contact structure on $L$, if 
	\begin{enumerate}
	\item $\mathcal{L}$ is a (complex) Dirac-Jacobi structure
	\item $\mathcal{L}\cap\cc{\mathcal{L}}=\{0\}$
	\end{enumerate}
\end{definition}
A generalized contact structure can be also seen as an endomorphism of $\mathbb{D}L$ of the form 
	\begin{align*}
	\begin{pmatrix}
	\phi & J^\sharp \\
	\alpha^\flat & \phi^* 
	\end{pmatrix},
	\end{align*}
where $\phi\in \mathrm{End}(DL)$, $J\in \Secinfty((J^1 L)^*\tensor L)$ and $\alpha\in \Omega_L^2(M)$ 
(see \cite{VW2016} and  \cite{2017arXiv171108310S}). This endomorphism has to fulfill certain properties: it has 
to be almost complex, compatible with the pairing and integrable, which we do not explain what it means here and refer the 
reader to \cite{VW2016}.  The $+\I$-Eigenbundle produces a generalized contact structure in the sense of Definition 
\ref{Def: GenCon}. Moreover, we have that among many more conditions that $J$ is a Jacobi structure. Let us now pick a 
(cosymplectic or cocontact) transversal to $J$ together with an Euler-like derivation $\Delta=J^\sharp (\alpha)$, then 
$(\Delta, \I\alpha-\phi^*(\alpha))\in\Secinfty(\mathcal{L})$.  With the techniques from Section \ref{Sec: NormFroDJ} and 
Section \ref{Sec: NormForJ}, one can show that 
	\begin{align*}
	\mathfrak{B}_\Psi(\mathcal{L})=\mathfrak{B}_{I\circ P}(\mathcal{L})^{\I\omega+\beta},
	\end{align*}
where $\omega=\int_{0}^{1}\frac{1}{t}(\Phi_{\log(t)}^\Delta)^*\D_L\alpha\D t$ and $\beta=-\int_{0}^{1}\frac{1}{t}
(\Phi_{\log(t)}^\Delta)^*\D_L\phi^*(\alpha)\D t$.
This is nothing else but a normal form for generalized contact bundles. This can be pushed more forward to prove a local splitting of generalized bundles, but this has already be done in \cite{2017arXiv171108310S} with similar techniques.

\begin{appendix}
\section{The Moser trick for Jacobi manifolds}\label{App: Moser}
Let  $J\in \Secinfty(\Anti^2 (J^1L)^*\tensor L)$ be a Jacobi structure on a line bundle $L\to M$. Moreover, we assume 
having  smooth family of closed 2-forms $\sigma_t$, such that $\sigma_0=0$ and $\mathcal{L}_J^{\sigma_t}$ is a Jacobi structure for all $t$,
denoted by $J_t$. 
For 
	\begin{align*}
	\alpha_t:=-\frac{\partial}{\partial t}\iota_\mathbbm{1} \sigma_t
	\end{align*}
the equation 
	\begin{align*}
	\frac{\partial}{\partial t}\sigma_t=-\D_L\alpha_t
	\end{align*}
holds.
We define the Moser-derivation by 
	\begin{align*}
	\Delta_t:= -J_t^\sharp(\alpha_t)
	\end{align*}
and its flow by $\Phi_t\in \Aut(L)$, where we assume it exists for on open subset containing $[0,1]$. 
Let us  compute
	\begin{equation}\label{Eq: Moser}
	\begin{aligned}
	\frac{\D}{\D t} \Phi^*_tJ_t &
	=\Phi_t^*([\Delta_t,J_t]+\frac{\D}{\D t} J_t)\\&
	=\Phi_t^*(-[J_t^\sharp(\alpha_t),J_t]+\frac{\D}{\D t} J_t)\\&
	=\Phi^*_t(J_t^\sharp(-\D_L\alpha_t)+\frac{\D}{\D t} J_t).
	\end{aligned}
	\end{equation}
It is easy to see that 
	\begin{align*}
	J_t^\sharp=J^\sharp\circ (\id+ \sigma_t^\flat \circ J^\sharp)^{-1}
	\end{align*}
and hence we can compute
	\begin{align*}
	\frac{\D}{\D t} J_t^\sharp &
	=\frac{\D}{\D t} J^\sharp\circ(\id+ \sigma_t^\flat \circ J^\sharp)^{-1}\\&
	=- J^\sharp\circ  (\id+ \sigma_t^\flat \circ J^\sharp)^{-1}\circ(\frac{\D}{\D t}(\id+ \sigma_t^\flat \circ J^\sharp))
	\circ (\id+ \sigma_t^\flat \circ J^\sharp)^{-1}\\&
	=- J^\sharp\circ  (\id+ \sigma_t^\flat \circ J^\sharp)^{-1}\circ
	((\frac{\partial}{\partial t}\sigma_t)^\flat\circ J^\sharp)
	\circ (\id+ \sigma_t^\flat \circ J^\sharp)^{-1}\\&
	=-J_t^\sharp \circ (\frac{\partial}{\partial t}\sigma_t)^\flat \circ J_t^\sharp\\&
	=(-J_t^\sharp(\frac{\partial}{\partial t}\sigma_t))^\sharp\\&
	= \big(J_t^\sharp(\D_L\alpha_t)\big)^\sharp, 
	\end{align*}	 
and hence $\frac{\D}{\D t} J_t=J_t^\sharp(\D_L\alpha_t)$. If we use this equality in Equation \ref{Eq: Moser}, we find 
	\begin{align*}
	\frac{\D}{\D t} \Phi^*_tJ_t=0,
	\end{align*}
so we finally have $J=\Phi^*_0J_0=\Phi_1^* J_1$ and hence the two Jacobi structures are isomorphic. 	
\end{appendix}
\bibliographystyle{plain}

\begin{thebibliography}{10}

\bibitem{B2017}
C.~Blohmann.
\newblock {Removable presymplectic singularities and the local splitting of
  Dirac structures}.
\newblock {\em IMRN}, (23).

\bibitem{BGG2017}
A.~J. Bruce, K.~Grabowska, and J.~Grabowski.
\newblock {Remarks on Contact and Jacobi Geometry}.
\newblock {\em SIGMA}, 13:059, 2017.

\bibitem{2016arXiv160505386B}
H.~{Bursztyn}, H.~{Lima}, and E.~{Meinrenken}.
\newblock {Splitting theorems for Poisson and related structures}.
\newblock {\em J. reine angew. Math.}, 2016.

\bibitem{CHEN2010799}
Z.~Chen and Z.-J. Liu.
\newblock Omni-lie algebroids.
\newblock {\em Journal of Geometry and Physics}, 60(5):799 -- 808, 2010.

\bibitem{SplittingThmJac}
Pierre Dazord, André Lichnerowicz, and Charles-Michel Marle.
\newblock Structure locale des variétés de jacobi. (local structure of jacobi
  manifolds).
\newblock {\em Journal de Mathématiques Pures et Appliquées. Neuvième
  Série}, 70, 1991.

\bibitem{FM2017}
P.~Frejlich and M{\u{a}}rcuț.
\newblock {The normal form theorem around Poisson transversals}.
\newblock {\em Pacific J. Math.}, 287(2), 2017.

\bibitem{FM2018}
P.~Frejlich and I.~M{\u{a}}rcuț.
\newblock On dual pairs in dirac geometry.
\newblock {\em Mathematische Zeitschrift}, 289(1):171--200, 2018.

\bibitem{GR2012}
J.Grabowski and M.~Rotkiewicz.
\newblock Graded bundles and homogeneity structures.
\newblock {\em Journal of Geometry and Physics}, 62(1):21 -- 36, 2012.

\bibitem{K1976}
A.~A. {Kirillov}.
\newblock {Local Lie algebras.}
\newblock {\em {Russ. Math. Surv.}}, 31(4):55--75, 1976.

\bibitem{2017arXiv171108310S}
J.~{Schnitzer} and L.~{Vitagliano}.
\newblock {The Local Structure of Generalized Contact Bundles}.
\newblock {\em ArXiv e-prints}, 2017.

\bibitem{2017arXiv170508962T}
A.~G. {Tortorella}.
\newblock {Deformations of coisotropic submanifolds in Jacobi manifolds}.
\newblock {\em ArXiv e-prints}, May 2017.

\bibitem{DirJacBun}
L.~{Vitagliano}.
\newblock {Dirac-Jacobi Bundles}.
\newblock {\em J.~Sympl.~Geom.}, 16(2), 2016.

\bibitem{VW2016}
L.~Vitagliano and A.~Wade.
\newblock Generalized contact bundles.
\newblock {\em Comptes Rendus Mathematique}, 354(3):313 -- 317, 2016.

\bibitem{Wade2000}
A.~Wade.
\newblock {Confromal Dirac Structures}.
\newblock {\em Lett. Math. Phys.}, 53:331–348, 2000.

\bibitem{weinstein1983}
A.~Weinstein.
\newblock The local structure of poisson manifolds.
\newblock {\em J. Differential Geom.}, 18(3):523--557, 1983.

\end{thebibliography}

\end{document}